\documentclass[a4paper, fleqn, reqno, 11pt]{amsart}
\usepackage{times,latexsym,amssymb, stmaryrd, amsmath, amsthm}
\usepackage[colorlinks,pdfpagelabels,pdfstartview = FitH,bookmarksopen = true,bookmarksnumbered = true,linkcolor = blue,plainpages = false,hypertexnames = false,citecolor = red]{hyperref}

\numberwithin{equation}{section}

\newtheorem{Theorem}{Theorem}[section]

\newtheorem{Lemma}[Theorem]{Lemma}

\newtheorem{Remark}[Theorem]{Remark}

\def\dz{\,dz}
\def\dx{\,dx}

\def\dt{\,dt}
\def\ds{\,ds}

\DeclareMathOperator{\divergence}{div}

\def\mean#1{\mathchoice%
         {\mathop{\kern 0.2em\vrule width 0.6em height 0.69678ex depth -0.58065ex
                 \kern -0.8em \intop}\nolimits_{\kern -0.4em#1}}%
         {\mathop{\kern 0.1em\vrule width 0.5em height 0.69678ex depth -0.60387ex
                 \kern -0.6em \intop}\nolimits_{#1}}%
         {\mathop{\kern 0.1em\vrule width 0.5em height 0.69678ex
             depth -0.60387ex
                 \kern -0.6em \intop}\nolimits_{#1}}%
         {\mathop{\kern 0.1em\vrule width 0.5em height 0.69678ex depth -0.60387ex
                 \kern -0.6em \intop}\nolimits_{#1}}}

\newcommand{\aveint}[2]{\mathchoice%
          {\mathop{\kern 0.2em\vrule width 0.6em height 0.69678ex depth -0.58065ex
                  \kern -0.8em \intop}\nolimits_{\kern -0.45em#1}^{#2}}%
          {\mathop{\kern 0.1em\vrule width 0.5em height 0.69678ex depth -0.60387ex
                  \kern -0.6em \intop}\nolimits_{#1}^{#2}}%
          {\mathop{\kern 0.1em\vrule width 0.5em height 0.69678ex depth -0.60387ex
                  \kern -0.6em \intop}\nolimits_{#1}^{#2}}%
          {\mathop{\kern 0.1em\vrule width 0.5em height 0.69678ex depth -0.60387ex
                  \kern -0.6em \intop}\nolimits_{#1}^{#2}}}

\newcommand{\loc}{\textnormal{loc}}

\newcommand{\R}{{\mathbb R}}

\newcommand{\N}{{\mathbb N}}

\newcommand{\Ph}{\mathcal{P}}

\newcommand{\vs}{\vspace{3mm}}
\newcommand{\LWp}{L^p(0,T;W^{1,p}(\Omega))}
\newcommand{\CL}{C([0,T];L^2(\Omega))}
\newcommand{\LWpz}{L^p(0,T;W^{1,p}_0(\Omega))}
\newcommand{\Ldual}{L^{p'}(0,T;W^{-1,p'}(\Omega))}
\newcommand{\OT}{{\Omega_T}}

\allowdisplaybreaks
\begin{document}

\title{Lorentz estimates for obstacle parabolic problems}

\vspace{5mm}

\author{Paolo Baroni}

\address{Paolo Baroni, Department of Mathematics, Uppsala University, L\"agerhyddsv\"agen 1, SE-751 06, Uppsala, Sweden} \email{paolo.baroni@math.uu.se}


\date{\scriptsize \today}

\begin{abstract}
We prove that the spatial gradient of (variational) solutions to parabolic obstacle problems of $p$-Laplacian type enjoys the same regularity of the data and of the derivatives of the obstacle in the scale of Lorentz spaces.
\end{abstract}

\maketitle

\section{Introduction}
In this paper we deal with the {\em obstacle problem} related to the parabolic Cauchy-Dirichlet problem
\begin{equation}\label{eq.parabolic}
\begin{cases}
u_t-\divergence a(x,t,Du)=f-\divergence\big[{|F|}^{p-2}F\big]\ &\text{in $\Omega_T=\Omega\times(0,T),$}\\[3pt]
u= 0\qquad&\text{on $\partial_{\rm lat}\Omega_T=\partial\Omega\times(0,T)$,}\\[3pt]
x(\cdot,0)=u_0&\text{in $\Omega$,}
\end{cases}
\end{equation}
where the vector field models the $p$-Laplacian operator with coefficients
\begin{equation}\label{plpal}
a(x,t,Du)\approx b(x,t)\big(s^2+|Du|^2\big)^{\frac{p-2}{2}}Du,\qquad p> \frac{2n}{n+2},\quad s\in[0,1], 
\end{equation}
see \eqref{assumptions}, and where the obstacle $\psi$ is not continuous, as often considered in the literature. We are interested in {\em sharp integrability estimates} for the gradient $Du$ of solutions to the {\em variational inequality} related to \eqref{eq.parabolic} in terms of integrability of the data on the right-hand side $f,F$ and of the obstacle $\psi$  in the {\em scale of Lorentz spaces}; here $\Omega\subset\R^n$, $n\geq2$ is a bounded domain and it will be so for the rest of the paper. More precisely, given an obstacle function $\psi:\Omega\times[0,T]\to\R$,
\begin{equation}\label{int.ostacolo}
\psi\in \LWp\cap\CL 
\end{equation}
such that 
\begin{equation}\label{reg.ostacolo}
\partial_t\psi\in L^{p'}(\Omega_T)\qquad\text{and}\qquad \psi\leq0\quad\text{a.e. on $\partial_{\rm lat}\Omega_T$}
\end{equation}
and functions
\begin{equation}\label{int.data}
F\in L^p(\Omega_T;\R^n)\qquad\text{and}\qquad f\in L^{p'}(\Omega_T)
\end{equation}
(with $p'$ we denote the H\"older conjugate of $p$, i.e., $p':=p/(p-1)$ for $p>1$), we consider functions $u\in K_0$, where 
\[
K_0:=\big\{u\in \LWpz\cap \CL:u\geq\psi\text{ a.e. in }\Omega_T\big\},
\]
satisfying the variational inequality
\begin{align}\label{inequality}
&\int_0^T\langle\partial_t v,v-u\rangle_{W^{-1,p}\times W^{1,p}_0}\dt+\int_{\Omega_T}\langle a(x,t,Du),Dv-Du\rangle\dz\notag\\
&\geq-\frac12\int_{\Omega}|v(\cdot,0)-u_0|^2\dx+\int_{\Omega_T}\langle |F|^{p-2}F,Dv-Du\rangle\dz\notag\\
&\hspace{7cm}+\int_{\Omega_T}f(v-u)\dz
\end{align}
for any function $v\in K_0'$, with
\[
K_0':=\big\{v\in K_0:\partial_tv\in \Ldual\};
\]
$\langle\cdot,\cdot\rangle_{W^{-1,p}\times W^{1,p}_0}$ denotes the duality pairing crochet between $W^{1,p}_0(\Omega)$ and its dual space $W^{-1,p}(\Omega)$, while $\langle\cdot,\cdot\rangle$ is the scalar product in $\R^n$. We immediately mention that existence and uniqueness for the problem we are considering can be inferred from \cite[Theorem 6.1]{BDM2}. For the initial value we shall assume
\begin{equation}\label{initial.datum}
u_0\in W_0^{1,p}(\Omega)\qquad\text{and}\qquad u_0\geq \psi(\cdot,0)\quad\text{a.e. in $\Omega$}; 
\end{equation}
using an approximation scheme, we can also allow for initial data in $u_0\in L^2(\Omega)$. The vector fields we treat model the $p$-Laplacian operator in the following sense: we take $a:\Omega_T\times\R^n\to\R^n$ such that $\partial_\xi a$ is a Carath\'eodory function and such that the following ellipticity and growth conditions are satisfied:
\begin{equation}\label{assumptions}
\begin{cases}
	\langle\partial_\xi a(x,t,\xi)\lambda,\lambda\rangle\geq\nu\big(s^2+|\xi|^2)^{\frac{p-2}{2}}|\lambda|^2,\\[5pt]
	|a(x,t,\xi)|+|\partial_\xi a(x,t,\xi)|\big(s^2+|\xi|^2\big)^{\frac12} \leq L\big(s^2+|\xi|^2\big)^{\frac{p-1}{2}},
 \end{cases}
\end{equation}
for almost every $(x,t)\in\Omega_T$ and all $\xi,\xi_1,\xi_2,\lambda\in\R^n$; the structural constants satisfy $0<\nu\leq 1\leq L<\infty$, $s\in[0,1]$ is the degeneracy parameter and the exponent $p$ will always satisfy the lower bound $p>\frac{2n}{n+2}$ as in \eqref{plpal}. Moreover we shall consider the following nonlinear VMO condition in the spirit of \cite{BW, KuusiMingione:2011}: defining for balls $ B\subset\Omega$ and for all $t\in (0,T)$ and all $\xi\in\R^n$ the averaged vector field
\begin{equation}\label{av.vf}
(a)_{B}(t,\xi):=\mean{B}a(\cdot,t,\xi)\dx, 
\end{equation}
we require the averaged, normalized modulus of oscillation $\omega_a(R)\in[0,2L]$
\begin{equation}\label{mod.osc}
\omega_a(R):=\sup_{\substack{t\in(0,T),\\B\in \mathcal B_R, \xi\in\R^n}}\biggl(\mean{B}\biggl(\frac{|a(y,t,\xi)-(a)_{B}(t,\xi)|}{{(s^2+|\xi|^2)}^{(p-1)/2}}\biggr)^2\,dy\biggr)^{\frac12}
\end{equation}
where $\mathcal B_R$ is the collection of balls $\{B\equiv B_r(x)\subset\Omega:0<r\leq R\}$, to satisfy
\begin{equation}\label{ass.omega}
\lim_{R\searrow0}\omega_a(R)=0. 
\end{equation}
This means that, if we consider the model case in \eqref{plpal} with product coefficients $b(x,t)=d(x)h(t)$, we can allow bounded and measurable time-coefficients ($h\in L^\infty(0,T)$) and bounded and VMO spatial ones ($d\in (L^\infty\cap VMO)(\Omega)$); this kind on ``non-linear VMO condition'' includes, as particular case, the regularity conditions we assumed in \cite{BL} for systems
. VMO regularity  {\em only with respect to the spatial variables} has been often assumed to prove regularity estimates of this kind, starting from \cite{KR, KZ}, in the case without obstacle; see also \cite{AM07, BW}. 

\vs

Finally we are in position to state the main result of our paper:
\begin{Theorem}\label{thm.main}
 Let $u\in K_0$ satisfy the variational inequality \eqref{inequality}, where the vector field $a$ satisfies \eqref{assumptions} and \eqref{ass.omega}; moreover suppose that
 \begin{equation}\label{def-Psi}
|D\psi|+|\partial_t\psi|^{1/(p-1)}+|F|+|f|^{1/(p-1)}\in L(\gamma,q)\quad\text{locally in $\Omega_T$}
\end{equation}
for some $\gamma>p$ and some $q\in(0,\infty]$. Then $|Du|\in L(\gamma,q)$ locally in $\Omega_T$ and there exists a radius $R_0\leq 1$, depending on $n,p,\nu,L,\omega_a(\cdot),\gamma$ and on $q$ in the case $q<\infty$, such that the following local estimate holds, for parabolic cylinders $Q_{2R}\equiv Q_{2R}(z_0)\subset\Omega_T$, with $R\leq R_0$:
\begin{multline}\label{main.est}
|Q_R|^{-\frac1\gamma}\big\||Du|+s\big\|_{L(\gamma,q)(Q_R)}\leq c\,\biggl(\mean{Q_{2R}}\big(|Du|+s\big)^p\dz\biggr)^{\frac dp}\\+c\,|Q_{2R}|^{-\frac d\gamma}{\|\Psi_{2R}+1\|}_{L(\gamma,q)(Q_{2R})}^d,
\end{multline}
where the function $\Psi_R$, belonging locally to $L(\gamma,q)(\Omega_T)$, is defined by
\begin{equation}\label{Psi}
\Psi_R:=|F|+|D\psi|+R^{1/(p-1)}\big(|f|^{1/(p-1)}+|\partial_t\psi|^{1/(p-1)}\big). 
\end{equation}
The constant in \eqref{main.est} depends on $n,p,\nu,L,\gamma,q$ (except in the case $q=\infty$, where it depends only on $n,p,\nu,L,\gamma$) and the scaling deficit $d\geq1$ is defined by
 \begin{equation}\label{def-d}
d\equiv d(p):=
\begin{cases}
\displaystyle{\frac{p}{2}}&\text{if $\quad p\geq2$,}\\[10pt]
\displaystyle{\frac{2p}{p(n+2)-2n}}\quad&\text{if $\quad\displaystyle{\frac{2n}{n+2}<p<2}$.}
\end{cases}
\end{equation}
Note that the constant $c$ depends critically on $\gamma-p$, in the sense that $c\to\infty$ when $\gamma\to p$.
\end{Theorem}
We recall that the Lorentz space $L(\gamma,q)(A)$, for $A\subset \R^k$, $k\in\N$, open set and for parameters $1\leq\gamma<\infty$ and $0<q<\infty$, is defined by requiring, for a measurable function $g:A\to\R$, that
\begin{equation}\label{def.Lorentz}
     \| g\|_{L(\gamma,q)(A)}^q := q \int_0^\infty \Big( \lambda^\gamma \big| \{\xi\in A:\,|g(\xi)| > \lambda\} \big| \Big)^\frac q\gamma \frac{d\lambda}{\lambda}  <\infty.
\end{equation}
If $q=\infty$, $1\leq\gamma<\infty$, the space $L(\gamma,\infty)(A)$ is by definition the Marcinkiewicz space $\mathcal M^\gamma(A)$, the space of measurable functions $g$ such that
\begin{equation}\label{def.Marcin}
     \| g\|_{L(\gamma,\infty)(A)}=\| g\|_{\mathcal M^\gamma(A)} := \sup_{\lambda>0}\Big( \lambda^{\gamma} \big| \{\xi\in A:\,|g(\xi)| > \lambda\} \big| \Big)^{\frac1\gamma}<\infty. 
\end{equation}
The local variant of such spaces is defined in the usual way; see Paragraph \ref{Lorentz} for some more details about Lorentz spaces. 

\vs

A few comments about our assumptions and our result. To start with, note that the lower bound for the exponent, analogous to that in \eqref{plpal}-\eqref{plap}, is unavoidable since it already naturally appears in the regularity theory of solutions to parabolic $p$-Laplacian operators (see \cite{DiBenedetto, Urbano, AM07, KMIbero, BOK}). 

Note also that the result is sharp, and this follows if we consider the regularity of solutions on the so-called coincidence set, i.e. that portion of the domain where the solutions and the obstacle coincide; if we consider {\em solutions} to \eqref{eq.parabolic}, the implication
\[
|F|\in L(\gamma,q)\quad\text{locally in $\Omega_T$}\quad\Longrightarrow\quad |Du|\in L(\gamma,q)\quad\text{locally in $\Omega_T$} 
\]
has been proved by the author in \cite{BL}. Our work essentially relies upon the work \cite{AM07} of Acerbi and Mingione, where the Lebesgue version $q=\gamma$ of Theorem \ref{thm.main} {\em without obstacle} has been proved:
\begin{equation}\label{eq.am07}
|F|\in L^\gamma_\loc(\Omega_T)\quad\Longrightarrow\quad |Du|\in L^\gamma_\loc(\Omega_T) 
\end{equation}
for $\gamma>p$. In this paper techniques to handle Calder\'on-Zygmund estimate for degenerate and singular parabolic systems of $p$-Laplacian type have been developed for the first time; see also \cite{BOK} for a version up to the boundary. These techniques have then been used extensively in the last year, for instance to obtain global estimates in domains with rough boundaries \cite{BOK, BW}; these tools (which we shall describe in a while) have also been shown to be flexible enough to handle parabolic (and elliptic) obstacle problem. We refer in particular to \cite{BDM2}, where the analogue of our Theorem \ref{thm.main} has been proved in the setting of Lebesgue spaces:
\begin{multline*}
|D\psi|+|\partial_t\psi|^{1/(p-1)}+|F|+|f|^{1/(p-1)}\in L^\gamma_\loc(\Omega_T)\\\Longrightarrow\qquad |Du|\in L^\gamma_\loc(\Omega_T) 
\end{multline*}
for $\gamma>p$. On the other hand in \cite{BL} the author proved the natural generalization of \eqref{eq.am07} to the Lorentz spaces setting, see the same \cite{BL} for further comments and references. In this paper we show how to modify the technique which lead to \eqref{eq.am07} in order to meet both the obstacle-structure of the problem and the setting of Lorentz spaces; moreover, while using extensively some of the results proved in \cite{BDM2}, we shall simplify some of the arguments: in particular we will not prove \eqref{main.est} as an {\em a priori estimate} for solution with {\em bounded gradient} (this will be needed to re-absorb certain terms appearing on the right-hand side), but we shall argue directly on truncations of the gradient, see \eqref{notation.k} and \eqref{notation.k.est}. Finally, with regard to obstacle problems, we want to mention the recent interesting paper \cite{LP} by Lindqvist \& Parviainen, where it is discussed the topic of existence of solutions for {\em irregular obstacle} problems, in the sense that obstacles do not even possess time derivative; a delicate interaction between regularity of the obstacle and the regularity of the test functions comes here into play.

\vs

The approach developed in \cite{AM07}, with elements from \cite{CaPe, KL}, is essentially based on the construction of an appropriate {\em family of intrinsic cylinders}  where the equation re-homogenize: already when considering the homogeneous evolutionary $p$-Laplace equation
\begin{equation}\label{plap}
\partial_tu-\divergence \big[|Du|^{p-2}Du\big]=0, \qquad p> \frac{2n}{n+2},
\end{equation}
one has to work {\em not} with the standard parabolic cylinders $Q_R(x_0,t_0):=B_R(x_0)\times (t_0-R^2,t_0)$, but with certain cylinders whose shape is devised to rebalance the lack of scaling of the equation: indeed the elliptic part is homogeneous of degree $p-1$, while the parabolic part is clearly of degree $1$, and this tells us that {\em no universal family of balls is associated to the equation}. As a consequence, typical harmonic analysis tools like maximal operators are automatically ruled out. One, hence, following DiBenedetto \cite{DiBenedetto, DiBenedettoFried} and considering here for simplicity in the case $p \geq 2$, works on cylinders of the type 
\[
Q^\lambda_R(z_0)\equiv Q^\lambda_R(x_0,t_0):=B_R(x_0)\times (t_0-\lambda^{2-p}R^2,t_0) 
\]
with $\lambda\geq1$ a scaling parameter; the heuristic underneath the choice of the scaling parameter $\lambda$ is the following. Suppose that on one of these cylinders the relation
\begin{equation}\label{introinsic}
\mean{Q^\lambda_R(z_0)}|Du|^p\dz\approx\lambda^p
\end{equation}
holds; we call such a cylinder {\em intrinsic}, since the parameter $\lambda$ appears both in the definition of the cylinder and in the values $Du$ takes over it and therefore every of these cylinders {\em depends explicitly on the solution}. Relation \eqref{introinsic} roughly tells that  $|Du| \approx \lambda$ on $Q_R^\lambda(z_0)$ and hence one may think to equation \eqref{plap} as actually $\partial_tu -\lambda^{p-2}\,\divergence Du=\partial_tu -\lambda^{p-2}\,\triangle u = 0$ in $Q_R^\lambda(z_0)$. Now, switching from the intrinsic cylinder $Q^\lambda_R(z_0)$ to $Q_1$, that is making the change of variables 
 \[
 v(x,t):=u(x_0 +Rx,t_0 +\lambda^{2-p}R^2t),\quad	(x,t)\in B_1 \times(-1,0)\equiv Q_1,
 \]
 we note that our equation finally rewrites as $\partial_tv-\Delta v=0$ in $Q_1$. This argument tells that on an {\em intrinsic cylinder} like \eqref{introinsic} the solution $u$ behaves as a solution to the heat equation. Note however that the previous argument is clearly only heuristic, and its implementation is far from straightforward; in particular it clearly has to be adapted to the problem we are dealing with, taking into account also the eventual presence of right-hand sides or obstacles, as in our case. Indeed our choice of intrinsic cylinders will be, see \eqref{exit.cylinder},
\[
 {\biggl(\mean{Q_R^\lambda(z_0)}\big(|Du|+s\big)^p\dz\biggr)}^{\frac 1p}+M{\biggl(\mean{Q_{R}^\lambda(z_0)}(\Psi_R+s)^\eta\dz\biggr)}^{\frac 1\eta}=\lambda,
\]
for some $\eta\in(p,\gamma)$ and a large constant $M$. This latter constant is essentially {\em the key point} in the approach of Acerbi and Mingione: the {\em weight} $M\gg1$
is a suitably chosen parameter, depending on the structural constants of the problem, which allow to quantitatively control the contribution of the data $f,F,\psi$.
Indeed, we know that whether 
\[
\mean{Q_R^\lambda(z_0)} \big(|Du|+s\big)^{p} \dz\approx\lambda^p\quad\text{or}\quad\mean{Q_{R}^\lambda(z_0)}(\Psi_R+s)^\eta\dz\approx\frac{\lambda^\eta}{M^\eta}
\]
holds. Therefore, again heuristically, or the equation is the non-degenerate as above in $Q^\lambda_R(z_0)$ or, if we choose $M$ large, $u$ solves (approximately) the $p$-Laplacian type equation
\[
    \partial_t u-\divergence a(x,t,Du)\approx0\qquad\text{on $Q^\lambda_R(z_0)$,}
\]
with constant, negative obstacle; this, a bit more formally, will be formalized in two steps: first we compare our variational solution to the solution to the Cauchy-Dirichlet problem, where the right-hand side has small as we please $L^{p'}$ norm
\[
 \begin{cases}
 \partial_tv-\divergence a(x,t,Dv)= \partial_t\psi-\divergence a(x,t,D\psi)&\text{in $Q_R^\lambda(z_0)$,}\\[3pt]
v\equiv u&\text{on $\partial_\Ph Q_R^\lambda(z_0)$,}
\end{cases}
\]
(see \eqref{f.compa.V} for the resulting comparison inequality) and then we compare in turn $v$ with the solution of the homogeneous problem
\[
 \begin{cases}
 \partial_t\tilde v-\divergence a(x,t,D\tilde v)= 0&\text{in $Q_{R/2}^\lambda(z_0)$,}\\[3pt]
\tilde v\equiv v&\text{on $\partial_\Ph Q_{R/2}^\lambda(z_0)$;}
\end{cases}
\]
the result is in \eqref{s.compa.V}. See also the analogue but somehow different heuristic explanation in \cite{BDM2}.

\section{Notation, function spaces and tools}
Here first we fix the notation we are going to use in this paper; moreover we shall collect some definitions and results regarding functional spaces we shall employ but also classic results for $p$-Laplacian type equations.
\subsection{Notation} 
The Euclidean space $\R^{n+1}$ will always be thought as $\R^n\times\R$, so a point $z\in\R^{n+1}$ will be often also denoted as $(x,t)$, $z_0$ as $(x_0,t_0)$ and so on. Being $B_R(x_0)$ the ball $\{x\in\R^n:|x-x_0|<R\}$, we shall consider parabolic cylinders of the form
\[
Q_R(z_0):=B_R(x_0)\times (t_0-R^2,t_0),
\] 
but we shall also deal with scaled cylinders of the form
\begin{equation}\label{cylinders}
     Q_R^{\lambda}(z_0):=
\begin{cases}
B_R(x_0)\times \big(t_0-\lambda^{2-p}R^2,t_0)\quad&\text{if $p\geq2$},\\[3pt]
B_{\lambda^{\frac{p-2}{2}}R}(x_0)\times \big(t_0-R^2,t_0)\quad&\text{if $p<2$},
\end{cases} 
\end{equation}
where the stretching parameter {\em will be always greater than one}: $\lambda\geq1$; hence in both cases $Q_R^\lambda(z_0)\subset Q_R^1(z_0)=Q_R(z_0)$. With $\chi B_R(x_0)$, for a constant $\chi>1$, we will denote the $\chi$-times enlarged ball, i.e. $\chi B_R(x_0):=B_{\chi R}(x_0)$, and the same for cylinders: $\chi Q_R^{\lambda}(z_0):=Q_{\chi R}^{\lambda}(z_0)$. In order to shorten notation, we shall denote $\Lambda_R^{\lambda}(t_0)=:\big(t_0-\lambda^{2-p}R^2,t_0)$ and $B_R^\lambda(x_0):=B_{\lambda^{\frac{p-2}{2}}R}(x_0)$, and we shall drop the $\lambda$ when it will be one: $ \Lambda_R(t_0)=:\big(t_0-R^2,t_0)$ and $B_R^\lambda(x_0):=B_R(x_0)$. Often we shall avoid to make explicit their centers in the following way: $Q_R^{\lambda}\equiv Q_R^{\lambda}(z_0)$ and similar.

\vs

Given $\tau\in (0,T)$ we shall write $\Omega_\tau$ for the cylinder $\Omega\times(0,\tau)$; by parabolic boundary of $\mathcal K:=C\times I$ in $\R^{n+1}$, we mean $\partial_{\Ph}\mathcal K:=C\times\{\inf I\}\cup \partial C\times I$. Being $A \in \R^k$ a measurable set with positive measure and $f:A \to \R^m$ an integrable map,  with $k,m \ge 1$, we indicate with ${(f)}_A$ the averaged integral
\begin{equation*}
      {(f)}_{A}:=\mean{A} f(\xi)\, d\xi := \frac{1}{|A|} \int_{A} f(\xi)\, d\xi.
\end{equation*}
We will denote with $c$ a generic constant always greater than one, possibly varying from line to line; however, the ones we shall need to recall will be denoted with special symbols, such as $c_{DiB}, \tilde c,c_\ell$. We finally remark that by $\sup$ we shall always mean essential supremum.

\subsection{Lorentz spaces}\label{Lorentz}
The reader might recall the definition of Lorentz spaces in \eqref{def.Lorentz}-\eqref{def.Marcin}. Since here we assume $A$ of finite measure, the spaces $L(\gamma,q)(A)$ decrease in the first parameter $\gamma$; this means that for $1\le \gamma_1\leq \gamma_2<\infty$ and $0<q\le\infty$ we have a continuous embedding $L(\gamma_2,q)(A)\hookrightarrow L(\gamma_1,q)(A)$ with 
\[
\| g\|_{L(\gamma_1,q)(A)} \le|A|^{\frac{1}{\gamma_1}-\frac{1}{\gamma_2}}\| g\|_{L(\gamma_2,q)(A)}.
\]
On the other hand the Lorentz spaces in general increase in the second parameter $q$, i.e. we have for $0<q_1\leq q_2\le\infty$ the continuous embedding $L(\gamma,q_1)(A)\hookrightarrow L(\gamma,q_2)(A)$ with
\[
\| g\|_{L(\gamma,q_2)(A)} \le c(\gamma,q_1,q_2)\| g\|_{L(\gamma,q_1)(A)}
\]
when $q_2<\infty$, while the constant clearly does not depend on $q_2$ when $q_2=\infty$; see, essentially, Lemma \ref{Stein} for $\lambda=0$ and an appropriate choice of the quantities involved. Note moreover that by Fubini's theorem we have
\[
    \| g\|_{L^\gamma(A)}^\gamma=\gamma\int_0^\infty \lambda^\gamma\big|\{ \xi\in A: |g(\xi)|>\lambda\}\big|\,\frac{d\lambda}{\lambda} =\| g\|_{L(\gamma,\gamma)(A)}^\gamma\, ,
\]
so that $L^\gamma (A)=L(\gamma,\gamma)(A)$. Finally we have  that $L(\gamma,q)(A)\subset L^\eta(A)$ for any $\eta<\gamma$ and all $0<q\leq\infty$, see for instance Lemma \ref{lab.Marci} together with the second embedding above.

\begin{Remark}\label{lsc}
{\rm Note that the notation we use might be misleading, since, due to the lack of sub-additivity, the quantity $\|\cdot\|_{L(\gamma,q)(A)}$ is just a quasi-norm.
Nevertheless, the mapping $g\mapsto \|g\|_{L(\gamma,q)(A)}$ is lower semi-continuous with respect to a.e. convergence, see \cite[Remark 3]{MinMat} or \cite[Section 3]{JB}. }
\end{Remark}
  
\subsection{Parabolic spaces}
We collect here some properties of parabolic Sobolev spaces, restricting to the case $p>\frac{2n}{n+2}$. First notice that the embedding $W^{1,p}_0(\Omega)\hookrightarrow L^2(\Omega)$ for such exponents and the identification given by Riesz's Theorem allows to identify
\begin{equation}\label{identification}
\langle v,u\rangle_{W^{-1,p}\times W^{1,p}_0}=\int_{\Omega}vu\dx\qquad \text{if $v\in L^2(\Omega)\subset W^{-1,p'}(\Omega)$,}
\end{equation}
for any $u\in W^{1,p}_0(\Omega)$, where, we recall, $\langle\cdot,\cdot\rangle_{W^{-1,p}\times W^{1,p}_0}$ is the duality pairing between $W^{1,p}_0(\Omega)$ and $W^{-1,p}(\Omega)$. 

\vs

The space $\Ldual$ is the space of functions $f$ (usually we will have/denote $f=\partial_tg$ for some function $g\in L^p(0,T;W^{1,p}(\Omega))$) such that $f\in W^{-1,p'}(\Omega)$ for a.e. $t\in (0,T)$ and moreover
\[
t\mapsto \|f(\cdot,t)\|_{W^{-1,p'}(\Omega)}\in L^{p'}((0,T));
\]
moreover note that the following implication holds
\begin{multline}\label{inclusion}
f\in \LWp\quad\text{and}\quad\partial_tf\in \Ldual \\\qquad \Longrightarrow \qquad f\in C([0,T];L^2(\Omega)),
\end{multline}
see \cite{Lions}. This means that we made redundant assumptions (as in \eqref{int.ostacolo}-\eqref{reg.ostacolo}), but we shall keep doing that, for the sake of clarity. For the next result, which allows to manipulate the parabolic part of the variational inequality, see \cite{Sho} or \cite[Lemma 2.1]{BDM2}.
\begin{Lemma}
 Let $u,v\in\LWp$ be such that 
 \[
	\partial_tu,\partial_tv\in\Ldual. 
\]
Then
\begin{equation}\label{lemma.1}
\langle\partial_tu(\cdot,t),u(\cdot,t) \rangle_{W^{-1,p'}\times W^{1,p}_0}=\frac12\int_{\Omega}|u(\cdot,t)|^2\dx 
\end{equation}
for every $t\in[0,T]$ and moreover the following integration by parts formula holds true:
\begin{multline}\label{lemma.2}
\int_0^T\langle\partial_tu(\cdot,t),v(\cdot,t) \rangle_{W^{-1,p'}\times W^{1,p}_0}\dt= \int_{\Omega}uv(\cdot,\tau)\dx\biggr|_{\tau=0}^T\\-\int_0^T\langle\partial_tv(\cdot,t),u(\cdot,t) \rangle_{W^{-1,p'}\times W^{1,p}_0}\dt.
\end{multline}
\end{Lemma}
Note that the previous result makes sense in light of \eqref{inclusion}.
\subsection{The $\boldsymbol V$-function} 
We introduce the  auxiliary vector field  $V_s : \R^n \to \R^n$ defined by
\[
V_s(z):=\big(s^2+|z|^2\big)^{\frac{p-2}{4}}z,
\]
 which is a locally Lipschitz bjiection from $\R^n$ into itself and which turns out to be very useful in particular to deal with monotonicity conditions related to $p$-Laplacian operator. Notice indeed that there holds
\[
\frac{1}{c_V}|\xi_1-\xi_2|^2\leq \frac{|V_s(\xi_1)-V_s(\xi_2)|^2}{(s^2+|\xi_1|^2+|\xi_2|^2)^{\frac{p-2}{2}}} \leq c_V|\xi_1-\xi_2|^2
\]
for all vectors $\xi_1,\xi_2 \in \R^n$ not simultaneously null if $s=0$ and for every $p>1$; the constant $c_V$ depends only on $n,p$. The previous inequality is relevant in manipulations involving the classic monotonicity estimate
\begin{multline*}
\langle a(x,t,\xi_1)-a(x,t,\xi_2), \xi_1-\xi_2 \rangle\\\geq\frac{1}{c(n,p,\nu)} \big(s^2+|\xi_1|^2+|\xi_2|^2\big)^{\frac{p-2}{2}}|\xi_1-\xi_2|^2 
\end{multline*}
for any $\xi_1,\xi_2 \in \R^n$ as above and with $p>1$, which in turn follows by \eqref{assumptions}$_1$ and which, at this point, can be rewritten as
\begin{equation}\label{monotonicity.V}
\langle a(x,t,\xi_1)-a(x,t,\xi_2), \xi_1-\xi_2 \rangle\geq\frac{1}{c_m(n,p,\nu)}\big|V_s(\xi_1)-V_s(\xi_2)\big|^2.
\end{equation}
Moreover the function $V_s$ can be used to rephrase a quite classical inequality, see \cite[Lemma 5]{AM07} and references therein.

\begin{Lemma}\label{fst.lemma}
Let $p>1$. Then there exists a constant $c_\ell\equiv c_\ell(n,p)$ such that for any $\xi_1$, $\xi_2\in\R^{n}$, not both zero, there holds
\[
\big(s+{|\xi_1|}\big)^p\leq c_\ell\big(s+{|\xi_2|}\big)^p+c_\ell\big|V_s(\xi_1)-V_s(\xi_2)\big|^2.
\]
\end{Lemma}
\subsection{Auxiliary results}
The following comparison principle has been proved in \cite[Lemma 2.8]{BDM2} in the case the vector field has no dependence on $(x,t)$. The proof in our case requires no modification, since the only assumption used in to treat the elliptic part is (the version without coefficients of) \eqref{monotonicity.V}.
\begin{Lemma}\label{comparison}
Suppose $v,\psi\in\LWp\cap\CL$ satisfy in the weak sense
\[
\begin{cases}
\partial_t\psi-\divergence a(x,t,D\psi)\leq \partial_tv-\divergence a(x,t,Dv)\qquad&\text{in $\OT$},\\[3pt]
\psi\leq v&\text{on $\partial_\Ph\Omega_T$,}
\end{cases}
\]
where $a$ satisfies \eqref{monotonicity.V}. Then $\psi\leq v$ almost everywhere in $\OT$.
\end{Lemma}

The following is the higher-integrability result for local solutions to parabolic $p$-Laplacian systems by Kinnunen and Lewis. We restate it for equations with zero right-hand side, including also minor modifications to adapt it to this situation. Note that in general is the better estimate we can expect for such equation, due to the low degree of regularity of the partial map $x\mapsto a(x,t,\xi)$.
\begin{Theorem}\label{teo.hh}
Let $\mathcal K=C\times I\subset\Omega_T$ and let $\tilde v\in L^p_\loc(I;W^{1,p}_\loc(C))$, $p>\frac{2n}{n+2}$, be a local weak solution to
\[
\tilde v_t-\divergence a(x,t,D\tilde v)=0\qquad\text{in $\mathcal K$.} 
\]
Then there exists a constant $\varepsilon_0>0$ depending on $n,p,\nu,L$, such that $D\tilde v \in L^{p(1+\varepsilon_0)}_\loc(\mathcal K)$ and moreover if $Q_{2R}^\lambda(z_0)\subset\mathcal K$ is a cylinder where the intrinsic relation
\[
\mean{Q_{2R}^\lambda}\big(s+|D\tilde v|\big)^p\dz\leq \kappa\lambda^p
\]
holds for some constant $\kappa\geq1$, then
\begin{equation}\label{hi.int.homo}
 \mean{Q_{R}^\lambda}\big(s+|D\tilde v|\big)^{p(1+\varepsilon)}\dz\leq c\,\lambda^{p(1+\varepsilon)}
\end{equation}
for any $\epsilon\in (0,\varepsilon_0]$ and for a constant $c\equiv c(n,p,\nu,L,\kappa)$.
\end{Theorem}
\begin{proof}
The proof follows in the case $p\geq2$ from \cite[Lemma 3]{AM07} and in the case $p<2$ from \cite[Lemma 4]{AM07}, with minor modification.
\end{proof}

Once known that the dependence of the vector field with respect of the spatial variable is more regular, one can expect gradient boundedness. The celebrated intrinsic $\sup$-bound for the gradient by DiBenedetto, see \cite[Chapter 8]{DiBenedetto}, \cite[Section 7]{KMIbero}, \cite{Min11}, is indeed encoded in the following
\begin{Theorem}\label{sup.Thm}
Let $\mathcal K=C\times I\subset\Omega_T$ and let $\tilde w\in L^p_\loc(I;W^{1,p}_\loc(C))$, $p>\frac{2n}{n+2}$, be a local weak solution to
\[
\partial_tw-\divergence \bar a(t,Dw)=0\qquad\text{in $\mathcal K$}
\]
where the vector field $\bar a:I\times\R^n\to\R^n$ satisfies \eqref{assumptions}, recast to the case with no $x$-dependence. Then $Du\in L^\infty_\loc(\mathcal K)$; moreover, if the cylinder $Q_{2R}^\lambda(z_0)\subset\mathcal K$ is such that
\[
\mean{Q_{2R}^\lambda(z_0)}\big(|Dw|+s\big)^p\dz\leq\kappa\lambda^p
\]
for some constant $\kappa\geq1$, then 
\[
\sup_{Q_R^\lambda(z_0)}|Dw|+s\leq c_{DiB}\lambda
\]
for a constant $c_{DiB}$ depending on $n,p,\nu,L,\kappa$.
\end{Theorem}
\subsection{Technical tools}
This first Lemma is the classic Hardy's inequality; see \cite[Theorem 330]{Hardy} or \cite{Hunt}.
\begin{Lemma}\label{Hardy}
Let $f:[0,+\infty)\to[0,+\infty)$ be a measurable function such that
\begin{equation}\label{well.posed}
 \int_0^\infty f(\lambda)\,d\lambda<\infty;
\end{equation}
then for any $\alpha\geq1$ and for any $r>0$ there holds
\[
 \int_0^\infty\lambda^r{\biggl(\int_\lambda^\infty f(\mu)\,d\mu\biggr)}^\alpha\frac{d\lambda}{\lambda}\leq\Bigl(\frac{\alpha}{r}\Big)^\alpha\int_0^\infty\lambda^r{\big[\lambda f(\lambda)\big]}^\alpha\frac{d\lambda}{\lambda}.
\]
\end{Lemma}

The following reverse-H\"older inequality is also classic; for its proof, see \cite[Appendix B.3]{Stein} for $\lambda=0$ or  \cite{BL}.
\begin{Lemma}\label{Stein}
Let $h:[0,+\infty)\to[0,+\infty)$ be a non-increasing, measurable function and let $\alpha_1\leq \alpha_2\leq\infty$ and $r>0$. Then, if $\alpha_2<\infty$
\[
\biggl[\int_\lambda^\infty\big[\mu^rh(\mu)\big]^{\alpha_2}\frac{d\mu}{\mu}\biggr]^{1/\alpha_2} \leq  \lambda^r h(\lambda)+c\,\biggl[\int_\lambda^\infty\big[\mu^rh(\mu)\big]^{\alpha_1}\frac{d\mu}{\mu}\biggr]^{1/\alpha_1}
\]
for any $\lambda\geq0$; if $\alpha_2=\infty$ then
\begin{equation}\label{Bram}
\sup_{\mu>\lambda}\big[\mu^rh(\mu)\big]\leq c\, \lambda^rh(\lambda)+c\,\biggl[\int_\lambda^\infty\big[\mu^rh(\mu)\big]^{\alpha_1}\frac{d\mu}{\mu}\biggr]^{1/\alpha_1}. 
\end{equation}
The constant $c$ depends only on $\alpha_1,\alpha_2,r$ except in the case $\alpha_2=\infty$. In this case $c\equiv c(\alpha_1,r)$.
\end{Lemma}

The following is a  a standard H\"older type inequality in Marcinkiewicz spaces; see \cite[Lemma 2.8]{Min}.
\begin{Lemma}\label{lab.Marci}
 Let $f\in \mathcal M^\gamma(A)$ for $A\subset\R^k$, $k\geq1$, of finite measure. Then
 \[
 \int_A|f|^\eta\dz\leq \frac{\gamma}{\gamma-\eta}|A|^{1-\frac{\eta}{\gamma}}\|f\|_{\mathcal M^\gamma(A)}^\eta,
 \]
 for any $\eta\in[1,\gamma)$.
\end{Lemma}

Finally, a very well-known iteration lemma.
\begin{Lemma}\label{it.lemma}
Let $\phi:[R,2R]\to[0,\infty)$ be a function such that
\[
	\phi(r_1)\leq\frac12\phi(r_2)+\mathcal{A}+\frac{\mathcal{B}}{(r_2-r_1)^\beta}
	\qquad\text{for every}\ R\leq r_1<r_2\leq2R,
\]
where $\mathcal{A},\mathcal{B}\geq0$ and $\beta>0$. Then
\[
	\phi(R)\leq c(\beta)\, \bigg[\mathcal{A}+\frac{\mathcal{B}}{R^\beta}\bigg].
\]
\end{Lemma}

\section{Problems with variable coefficients}
We collect in this section some results regarding the variational inequality \eqref{inequality}. First we show how it can be localized in time.
\begin{Lemma}[Localization]
Let $u\in K_0$ satisfy the variational inequality \eqref{inequality} for every $v \in K_0'$, with the obstacle $\psi$ and the data $F,f$ satisfying \eqref{int.ostacolo} to \eqref{int.data}. Then for every $\tau\in(0,T)$ and for every 
\begin{multline*}
 \tilde v\in \widetilde{K}_0':=\big\{\tilde v\in L^p(0,\tau;W^{1,p}_0(\Omega)):\\\tilde v\geq\psi\text{ a.e. in }\Omega_\tau\quad\text{and}\quad\partial_t\tilde v\in L^{p'}(0,\tau;W^{-1,p'}(\Omega))\big\},
\end{multline*}
(see \eqref{inclusion}), we have
\begin{multline}\label{inequality.localized}
\int_0^\tau\langle\partial_t \tilde v,\tilde v-u\rangle_{W^{-1,p}\times W^{1,p}_0}\dt+\int_{\Omega_\tau}\langle a(x,t,Du),D\tilde v-Du\rangle\dz\\
\geq-\frac12\int_{\Omega}|\tilde v(\cdot,0)-u_0|^2\dx+\int_{\Omega_\tau}\langle |F|^{p-2}F,D\tilde v-Du\rangle\dz\\+\int_{\Omega_\tau}f(\tilde v-u)\dz;
\end{multline}
we recall that $\Omega_\tau=\Omega\times(0,\tau)$.
\end{Lemma}
\begin{proof}
This proof is just technical, since we have to show how to appropriately choose a test function $v\in K_0'$ in \eqref{inequality} to get \eqref{inequality.localized}; in particular we want to choose $v=u$ in $\Omega_T\smallsetminus\Omega_\tau$ and this poses some difficulties, since we don't know whether $\partial_t u$ exists. Hence an appropriate approximation should be considered: we hence define, following \cite{ted}, for $h\in(0,T]$, $t\in(0,T]$ and  $u_0$ as in \eqref{initial.datum}, the mollification
\[
\llbracket u\rrbracket_h(\cdot,t):=e^{-\frac th}u_0(\cdot)+\frac1h\int_0^t e^{\frac{s-t}{h}}v(\cdot,s)\ds
\]
and moreover we take
\[
u_h:=\max\{\llbracket u\rrbracket_h,\psi\},
\]
being $\psi$ the obstacle. In \cite{ted}, see also \cite{BDM2, KLq} for other details, it is proved that $u_h\in K_0'$ and in particular 
\[
\partial_tu_h\in \Ldual\cap L^{\min\{2,p'\}}(\OT),
\] 
that $u_h\to u$ in $\LWp$ and in $L^2(\OT)$ as $h\to0$ and $u_h(\cdot,0)=u_0$ in the $L^2$ sense.

\vs

Now, for $h\in(0,T-\tau]$ take a decreasing Lipschitz cut-off function in time $\zeta\equiv \zeta_\epsilon\in W^{1,\infty}(\R)$ such that $0\leq\zeta_\epsilon\leq 1$, $\zeta_\epsilon\equiv1$ on $[0,\tau-\epsilon]$, $\zeta_\epsilon\equiv 0$ on $[\tau,T]$ and $ \zeta_\epsilon'=1/\epsilon$ in $(\tau-\epsilon,\tau)$ and use as test function in \eqref{inequality} 
\[
v\equiv v_{h,\epsilon}=\zeta_\epsilon\tilde v+(1-\zeta_\epsilon)u_h\in K_0'\qquad\text{in $\Omega_T$}.
\]
Notice indeed that since both $\tilde v$ and $u_h$ stay above $\psi$, then also $v$ does; moreover, $v_{h,\epsilon}\equiv \tilde v$ in $\Omega_{\tau-\epsilon}$, $v_{h,\epsilon}\equiv u_h$ in $\Omega_T\smallsetminus\Omega_\tau$. Hence the function $v_{h,\epsilon}$ can be used into \eqref{inequality} and this yields
\begin{multline}\label{long}
\int_0^{\tau-\epsilon}\langle\partial_t \tilde v,\tilde v-u\rangle_{W^{-1,p}\times W^{1,p}_0}\dt+\int_{\tau-\epsilon}^\tau\langle\partial_t v_{h,\epsilon},v_{h,\epsilon}-u\rangle_{W^{-1,p}\times W^{1,p}_0}\dt\\
+\int_\tau^T\langle\partial_t u_h,u_h-u\rangle_{W^{-1,p}\times W^{1,p}_0}\dt+\int_{\Omega_{\tau-\epsilon}}\langle a(x,t,Du),D\tilde v-Du\rangle\dz\\
+\int_{\Omega\times (\tau-\epsilon,\tau)}\langle a(x,t,Du),Dv_{h,\epsilon}-Du\rangle\dz\\+\int_{\Omega\times(\tau,T)}\langle a(x,t,Du),Du_h-Du\rangle\dz\\
\geq-\frac12\int_{\Omega}|\tilde v(\cdot,0)-u_0|^2\dx+\int_{\Omega_{\tau-\epsilon}}\langle |F|^{p-2}F,D\tilde v-Du\rangle\dz\\
+\int_{\Omega\times (\tau-\epsilon,\tau)}\langle |F|^{p-2}F,Dv_{h,\epsilon}-Du\rangle\dz\\+\int_{\Omega\times (\tau,T)}\langle |F|^{p-2}F,Du_h-Du\rangle\dz\\
 +\int_{\Omega_{\tau-\epsilon}}f(\tilde v-u)\dz +\int_{\Omega\times(\tau-\epsilon,\tau)}f(v_{h,\epsilon}-u)\dz +\int_{\Omega\times(\tau,T)}f(u_h-u)\dz.
\end{multline}
In the display above, first we want to let $\epsilon\searrow0$. The first and the fourth term on the left-hand side converge, respectively, to the corresponding ones over $(0,\tau)$ and $\Omega_\tau$; the same happens for the second and the fifth on the right-hand side. For the other ones, using the explicit expression for $v_{h,\epsilon}$ and triangle inequality, and also that $|\zeta_\epsilon|\equiv |\zeta_\epsilon(t)|\leq1$ yields
\begin{multline*}
\lim_{\epsilon\searrow0}\int_{\Omega\times (\tau-\epsilon,\tau)}\langle a(x,t,Du),Dv_{h,\epsilon}-Du\rangle\dz\\
\leq  \int_{\Omega\times (\tau-\epsilon,\tau)}\big(s^2+|Du|^2\big)^{\frac{p-1}{2}}\Big[|Du|+|Du_h|+|D\tilde v|\Big]\dz\xrightarrow[\epsilon\searrow0]{} 0;
\end{multline*}
similarly for the sum
\[
\int_{\Omega\times (\tau-\epsilon,\tau)}\langle |F|^{p-2}F,Dv_{h,\epsilon}-Du\rangle\dz+\int_{\Omega\times(\tau-\epsilon,\tau)}f(v_{h,\epsilon}-u)\dz\xrightarrow[\epsilon\searrow0]{} 0,
\]
Finally, the most problematic one: we split
\begin{align*}
\int_{\tau-\epsilon}^\tau&\langle\partial_t v_{h,\epsilon},v_{h,\epsilon}-u\rangle_{W^{-1,p}\times W^{1,p}_0}\dt \\
&=\int_{\tau-\epsilon}^\tau\langle\partial_t (v_{h,\epsilon}-u_h),v_{h,\epsilon}-u_h\rangle_{W^{-1,p}\times W^{1,p}_0}\dt\\
&\qquad+\int_{\tau-\epsilon}^\tau\langle\partial_t (v_{h,\epsilon}-u_h),u_h-u\rangle_{W^{-1,p}\times W^{1,p}_0}\dt\\
&\qquad+\int_{\tau-\epsilon}^\tau\langle\partial_t u_h,v_{h,\epsilon}-u_h\rangle_{W^{-1,p}\times W^{1,p}_0}\dt\\
&\qquad+\int_{\tau-\epsilon}^\tau\langle\partial_t u_h,u_h-u\rangle_{W^{-1,p}\times W^{1,p}_0}\dt=:I+II+III+IV.
\end{align*}
The first one is estimated as follows: taking into account that $v_{h,\epsilon}-u_h=\zeta_\epsilon(\tilde v-u_h)$,
\begin{multline*}
I=\int_{\tau-\epsilon}^\tau\frac12\frac{d}{dt}\int_\Omega \big|v_{h,\epsilon}(\cdot,t)-u_h(\cdot,t)\big|^2\dx\dt \\
=\frac12\int_\Omega \big[|\tilde v-u_h|^2\zeta_\epsilon^2\big](\cdot,\tau)\dx-\frac12\int_\Omega \big[|\tilde v-u_h|^2\zeta_\epsilon^2\big](\cdot,\tau-\epsilon)\dx\leq0
\end{multline*}
the first estimate follows by \eqref{lemma.1}, while the last one is due to the fact that $\zeta_\epsilon(\tau)=0$. For $III$, taking into account that $\zeta_\epsilon(t)$ is zero in $t=\tau$ and one in $t=\tau-\epsilon$, integrating by parts (in the sense specified by \eqref{lemma.2}) after recalling again \eqref{identification}:
\begin{multline*}
III=\int_{\Omega\times(\tau-\epsilon,\tau)}\langle\partial_t(v_{h,\epsilon}-u_h),u_h\rangle_{W^{-1,p}\times W^{1,p}_0}\dt\\-\int_{\Omega}\big[(\tilde v-u_h)u_h\big](\cdot,\tau-\epsilon)\dx.
\end{multline*}
Adding now $II$ and $III$ and recalling again that 
\[
-\partial_t(v_{h,\epsilon}-u_h)=-\zeta_\epsilon\partial_t(\tilde v-u_h)-\partial_t\zeta_\epsilon(\tilde v-u_h)=-\partial_t(\tilde v-u_h)\zeta_\epsilon+\frac{\tilde v-u_h}{\epsilon},
\]
we infer
\begin{align*}
II+III&=-\int_{\Omega\times(\tau-\epsilon,\tau)}\langle\partial_t(v_{h,\epsilon}-u_h),u\rangle_{W^{-1,p}\times W^{1,p}_0}\dt\\
&\hspace{4.5cm}-\int_{\Omega}\big[(\tilde v-u_h)u_h\big](\cdot,\tau-\epsilon)\dx\\
&= -\int_{\Omega\times(\tau-\epsilon,\tau)}\langle\partial_t(\tilde v-u_h),u\rangle_{W^{-1,p}\times W^{1,p}_0}\dt\\
&\quad+\aveint{\tau-\epsilon}{\tau}\int_{\Omega}(\tilde v-u_h)u\dx\dt-\int_{\Omega}\big[(\tilde v-u_h)u_h\big](\cdot,\tau-\epsilon)\dx\\
&\xrightarrow[\epsilon\to0]{} \int_{\Omega}\big[(\tilde v-u_h)u\big](\cdot,\tau)\dx-\int_{\Omega}\big[(\tilde v-u_h)u_h\big](\cdot,\tau)\dx,
\end{align*}
since $(\tilde v-u_h)u,(\tilde v-u_h)u_h\in C([0,T];L^2(\Omega))$. Finally we also have $IV\to 0$ as $\epsilon\searrow0$. Now, taking the limit $\epsilon\searrow0$ in \eqref{long}, we get
\begin{multline}\label{long2}
\int_0^\tau\langle\partial_t \tilde v,\tilde v-u\rangle_{W^{-1,p}\times W^{1,p}_0}\dt+\int_{\Omega}\big[(\tilde v-u_h)u\big](\cdot,\tau)\dx\\
-\int_{\Omega}\big[(\tilde v-u_h)u_h\big](\cdot,\tau)\dx+\int_\tau^T\langle\partial_t u_h,u_h-u\rangle_{W^{-1,p}\times W^{1,p}_0}\dt\\
+\int_{\Omega_\tau}\langle a(x,t,Du),D\tilde v-Du\rangle\dz+\int_{\Omega\times(\tau,T)}\langle a(x,t,Du),Du_h-Du\rangle\dz\\
\geq-\frac12\int_{\Omega}|\tilde v(\cdot,0)-u_0|^2\dx+\int_{\Omega_\tau}\langle |F|^{p-2}F,D\tilde v-Du\rangle\dz\\
+\int_{\Omega\times (\tau,T)}\langle |F|^{p-2}F,Du_h-Du\rangle\dz+\int_{\Omega_\tau}f(\tilde v-u)\dz \\
 +\int_{\Omega\times(\tau,T)}f(u_h-u)\dz.
\end{multline}
To conclude, we want to take the $\limsup$ as $h\searrow0$ in the previous inequality. Note that by the convergence of $u_h$ to $u$ in $L^p(0,T;W^{1,p}(\Omega))$. 
\[
\int_{\Omega\times(\tau,T)}\langle a(x,t,Du)-|F|^{p-2}F,Du_h-Du\rangle\dz+\int_{\Omega_\tau}f(\tilde v-u)\dz\xrightarrow[h\to0]{}0
\]
and moreover, by \eqref{identification} and (minor modifications of) \cite[Lemma 2.5]{BDM2} we know that
\begin{multline*}
 \limsup_{h\searrow0}\int_\tau^T\langle\partial_t u_h,u_h-u\rangle_{W^{-1,p}\times W^{1,p}_0}\dt\\=\limsup_{h\searrow0}\int_{\Omega\times(\tau,T)}\partial_t u_h(u_h-u)\dz\leq 0.
\end{multline*}
To conclude, 
\[
\limsup_{h\searrow0}\int_{\Omega}\big[(\tilde v-u_h)u\big](\cdot,\tau)\dx-\int_{\Omega}\big[(\tilde v-u_h)u_h\big](\cdot,\tau)\dx=0
\]
since $u_h\to u$ in $L^2(\OT)$. Indeed, up to a sub-sequence, $\|[u_h-u](\cdot,\tau)\|_{L^2(\Omega)}\to0$ for almost every $\tau\in[0,T]$; since $\tau\to\| [u_h-u](\cdot,\tau)\|_{L^2(\Omega)}$ is continuous over $[0,T]$, then convergence actually takes place everywhere. Putting all these informations into \eqref{long2} finally gives \eqref{inequality.localized}.
\end{proof}

We shall later need a higher integrability-type result for variational solutions to \eqref{inequality}; the following one has been proved by B\"ogelein and Scheven in \cite{BS}. We show the minor modifications that have to to be done with respect to their proof in order to get the following formulation.
\begin{Theorem}\label{hiint.BS}
Let $u\in L^p(0,T;W^{1,p}(\Omega))$ satisfy the variational inequality \eqref{inequality}, where the vector field satisfies \eqref{monotonicity.V} and
\[
	|a(x,t,\xi)| \leq L\big(s^2+|\xi|^2\big)^{\frac{p-1}{2}};
\]
moreover suppose that $F, |D\psi|\in L_\loc^{p(1+\sigma)}(\Omega_T)$ and $f,\partial_t\psi\in L_\loc^{p'(1+\sigma)}(\Omega_T)$ for some $\sigma>0$. Then there exists a constant $\varepsilon_1\in(0,\sigma]$ depending on $n,p,\nu,L,\sigma$, such that $|Du| \in L^{p(1+\varepsilon_1)}_\loc(\Omega_T)$ and moreover if $Q_{2R}^\lambda(z_0)\subset\Omega_T$ is a cylinder where the intrinsic bound
\begin{equation}\label{int.lemma}
\mean{Q_{2R}^\lambda}\big(s+|Du|\big)^p\dz+\biggl(\mean{Q_{2R}^\lambda}\Psi_{2R}^{p(1+\varepsilon_1)}\dz\biggr)^{1/(1+\varepsilon_1)}\leq \kappa\lambda^p
\end{equation}
holds for some constant $\kappa\geq1$, where $\Psi_R$ has been defined in \eqref{Psi}, then
\begin{equation}\label{hi.int.homo.var}
 \mean{Q_{R}^\lambda}\big(s+|Du|\big)^{p(1+\varepsilon_1)}\dz\leq c\,\lambda^{p(1+\varepsilon_1)}
\end{equation}
for a constant $c\equiv c(n,p,\nu,L,\sigma,\kappa)$.
\end{Theorem}
\begin{proof}
We cannot prove the local estimate \eqref{hi.int.homo.var} using the rescaling argument employed in \cite[Lemma 3-4]{AM07}, 
since we cannot localize \eqref{inequality}: assumptions on the boundary data in \cite{BS} are not usually satisfied locally by the solution we want to rescale. Therefore, rather than facing a technical regularization process, we prefer to proceed in a direct way by only showing the modifications to be done in the proof of \cite[Lemma 4.1]{BS}. In particular we want here to show that, under the assumptions of the Theorem, if $Q_{2R}^\lambda\equiv Q_{2R}^\lambda(z_0)\subset\Omega_T$, then the following estimate holds:
\begin{multline}\label{ppp}
\mean{Q_R^\lambda}\big(s+|Du|\big)^{p(1+\varepsilon)}\dz\leq c\,\lambda^{p(1+\varepsilon)}\\+c\,\lambda^{(1-d)p\varepsilon}\biggl(\mean{Q_{2R}^\lambda}\Big[\big(s+|Du|\big)^p+\Psi^p_{2R}\Big]\dz\biggr)^{1+\varepsilon d}+c\,\mean{Q_{2R}^\lambda}\Psi_{2R}^{p(1+\varepsilon)}\dz
\end{multline}
for all $\varepsilon\in(0,\varepsilon_1]$ and some constant $c$, as in the statement, with the scaling deficit defined in \eqref{def-d}. If moreover \eqref{int.lemma} holds, from the previous display it is immediate to get \eqref{hi.int.homo.var}. Since the procedure we are going to implement is very similar to that we shall describe in detail in Section \ref{proof.section}, we shall be very brief and we shall also write the major points, clearly for the arguments which do not need modifications.


\vspace{3mm}

We fix $Q_{2R}^\lambda\subset \Omega$, in both cases $p\geq2$ and $p<2$, and we consider the nested cylinders 
\[
Q_R^\lambda\subset Q_{r_1}^\lambda\subset Q_{r_2}^\lambda\subset Q_{2R}^\lambda,\qquad\text{with} \qquad R\leq r_1<r_2\leq2R.
\]
We also fix the quantity 
\[
\mu_0^{\frac pd}:=\lambda^{(1-d)\frac pd}\biggl(\mean{Q_{2R}^\lambda}\Big[\big(s+|Du|\big)^p+\Psi_{2R}^p\Big]\dz+\lambda^p\biggr)\geq\lambda^{\frac pd}
\]
and we consider points $z\in Q_{r_1}^\lambda$ and cylinders $Q_\varrho^\mu(z)$. Note that in the case $p<2$ we are considering here cylinders as defined in \eqref{cylinders}, differently from \cite{BS}; we shall however show the modification that should be done, taking into account that in any case $Q_\varrho^\mu(z)\subset Q_\varrho(z)$, since we are going to consider $\mu\geq1$. Notice that also our notation is slightly different from that of \cite{BS}: we indeed denote with $Q_\varrho^\mu(z)$ the (intrinsic) cylinders which play the role of the $Q_s^{(\lambda)}(z_0)$ in \cite[Section 4]{BS}. We consider here
\[
\mu>  B\mu_0, \quad\text{where}\quad B^{\frac pd}:=\biggl(\frac{160R}{r_2-r_1}\biggr)^N;\qquad \frac{r_2-r_1}{80}\leq\varrho\leq r_2-r_1
\]
and the important point is to note that, for such radii $\varrho$, $Q_\varrho^\mu(z)\subset Q_{2R}^\lambda$ if $z\in Q_{R}^\lambda$ and $\mu\geq\lambda$, since $B\geq1$. Hence, defining the operator
\[
CZ\big(Q_\varrho^\mu(z)\big):=\mean{Q_\varrho^\mu(z)}\Big[\big(s+|Du|\big)^p+\Psi_{2R}^p\Big]\dz
\]
we can estimate as in \cite[Step 1, pag. 951]{BS}, enlarging the domain of integration
\[
CZ\big(Q_\varrho^\mu(z)\big)\leq\frac{|Q_{2R}^\lambda|}{|Q_\varrho^\mu(z)|}\lambda^{(d-1)\frac pd}\mu_0^{\frac pd}< \frac{|Q_{2R}^\lambda|}{|Q_\varrho^\mu(z)|}\lambda^{(d-1)\frac pd}\mu^{\frac pd}B^{-\frac pd}.
\]
In both the cases $p\geq2$ and $p<2$ the right-hand side is bounded by $\mu^p$: indeed when $p\geq2$
\[
\frac{|Q_{2R}^\lambda|}{|Q_\varrho^\mu(z)|}\lambda^{(d-1)\frac pd}\mu^{\frac pd}B^{-\frac pd}=\Bigl(\frac\mu\lambda\Bigr)^{p-2}\Bigl(\frac{2R}{\varrho}\Bigr)^N\lambda^{p-2}\mu^2\biggl(\frac{160R}{r_2-r_1}\biggr)^{-N}\leq \mu^p
\]
recalling that $d=p/2$ in this case and that $1/\varrho\leq 80/(r_2-r_1)$; if $p<2$
\begin{multline*}
\frac{|Q_{2R}^\lambda|}{|Q_\varrho^\mu(z)|}\lambda^{(d-1)\frac pd}\mu^{\frac pd}B^{-\frac pd}\\=\Bigl(\frac\mu\lambda\Bigr)^{\frac{2-p}{2}n}\Bigl(\frac{2R}{\varrho}\Bigr)^N\lambda^{-\frac{p-2}{2}n}\mu^{p\frac{n+2}{n}-n}\biggl(\frac{160R}{r_2-r_1}\biggr)^{-N}
 \end{multline*}
and the last quantity is again bounded by $\mu^p$; recall that now $d=2p/[p(n+2)-2n]$. At this point the proof continues as in \cite{BS}: if $|Du(z)|+s>\mu$, then by Lebesgue's differentiation Theorem we have that $CZ(Q_\varrho^\mu(z))>\mu$ for small radii $0<\varrho\ll1$ and by absolute continuity we find a critical radius $\varrho_z< (r_2-r_1)/80$ such that $CZ(Q_{\varrho_z}^\mu(z))=\mu$. Note again that $Q_{80\varrho_z}^\mu(z)\subset Q_{2R}^\lambda$ and we slightly changed the super-level sets in play. Now the proof goes on exactly as after equation (4.8) in \cite{BS}, just keeping into account that there the $\Psi$ function does not include the radius $R$; this is to say, calling $\widetilde \Psi$ the function therein appearing, that $\widetilde \Psi=\Psi_1$, where $\Psi_R$ has been defined in \eqref{Psi}. This change, on the other hand, does not prevent to have the reverse H\"older's inequality of \cite[Lemma 3.1]{BS} also in our setting: that is,
\begin{multline*}
\mean{Q_{\varrho_z}^\mu(z)}(|Du|+s)^p\dz\\\leq c\,\biggl(\mean{Q_{2\varrho_z}^\mu(z)}(|Du|+s)^q\dz\biggr)^{\frac pq}+c\,\mean{Q_{8\varrho_z}^\mu(z)}\Psi_{8\varrho_z}^p\dz, 
\end{multline*}
where $q\equiv q(n,p)<p$ and $c\equiv c(n,p,\nu,L)$. Indeed, at a certain point (see in particular the estimates after (3.11)), in \cite{BS} the authors estimate $\Psi_{8\varrho_z}\leq \Psi_{8}=c(p)\widetilde \Psi$ since their radii $\varrho_z$ are smaller than one. At this point the proof, that is mostly algebraic and does not take into account the different expression for the cylinders we have, goes exactly as in \cite{BS} until the end of Section 4, once taking into account the aforementioned different meaning of the quantities into play; we have just to stress that, at a certain point of the proof, after the covering argument, we have to pointwise estimate $\Psi_{8\varrho_z}\leq \Psi_{2R}$. Hence, the application of Lemma \ref{it.lemma}, together with a truncation argument similar to that we are going to use at the end of Paragraph \ref{level.set}, leads to
\begin{multline*}
\mean{Q_R^\lambda}\big(s+|Du|\big)^{p(1+\varepsilon)}\dz\\\leq c\,\biggl(\mu_0^{\varepsilon p}\mean{Q_{2R}^\lambda}\big(s+|Du|\big)^p\dz+\mean{Q_{2R}^\lambda}\Psi^{p(1+\varepsilon)}_{2R}\dz\biggr)
 \end{multline*}
for any $\varepsilon\leq\varepsilon_1$, $\varepsilon_1$ described in \cite[Page 957]{BS}. Recalling the definition of $\mu_0$ it is immediate now to see that \eqref{ppp} follows. Finally note that taking $\lambda=1$ gives exactly back the result and the proof of \cite{BS}.
\end{proof}

\section{The proof of the Theorem}\label{proof.section}
Fix $Q_R(z_0)$ as in the statement of the theorem, such that $Q_{2R}(z_0)\subset\OT$ and $2R\leq R_0$; at this point of the proof we fix $R_0\equiv 1$, but in a subsequent step we shall reduce it in order to satisfy certain smallness conditions and this will cause the dependence stated in the theorem. For $d\geq1$ defined in \eqref{def-d} and $\Psi$ given by \eqref{Psi}, and for $M\geq1$ to be fixed later (only depending on $n,p,\nu,L,\gamma$ and possibily $q$), define the quantity
\begin{multline}\label{lambda0}
\lambda_0:={\biggl(\mean{Q_{2R}(z_0)}\big(|Du|+s\big)^p\dz\biggr)}^{\frac dp}\\+M^d{\biggl(\mean{Q_{2R}(z_0)}(\Psi_{2R}+s)^{\eta}\dz\biggr)}^{\frac d\eta}+1. 
\end{multline}
where $\eta:=p(1+\varepsilon_1)$, $\varepsilon_1\equiv \varepsilon_1(n,p,\nu,L,\gamma)$ being the higher integrability exponent given by Theorem \ref{hiint.BS} for the choice $\sigma=(\gamma-p)/(2p)$ (that yields $p(1+\sigma)=(p+\gamma)/2$); note indeed that $\Psi_{2R}\in L^{p(1+\sigma)}_\loc(\Omega_T)\subset L^{p(1+\varepsilon_1)}_\loc(\Omega_T)$ by \eqref{def-Psi} and the facts described in Paragraph \ref{Lorentz} about inclusions between Lorentz spaces. Notice moreover that this choice fixes a little imprecision in \cite{BL}; this value should replace the not correct one \cite{BL}; on the other hand, the whole proof does not require essentially any other change.

\vs

Consider now two intermediate radii $r_1,r_2$ such that $R\leq r_1<r_2\leq 2R$ and consider
\begin{equation}\label{Bradi}
 B:=2^d\Bigl(\frac{80R}{r_2-r_1}\Bigr)^{\frac Np d},\qquad\qquad \frac{r_2-r_1}{40}\leq r\leq r_2-r_1.
\end{equation}
To begin, we prove that for points $z\in Q_{r_1}$, levels $\lambda> B\lambda_0$ and radii as in \eqref{Bradi}, we have
\begin{multline}\label{sabove}
CZ\big(Q_r^\lambda(z)\big):={\biggl(\mean{Q_r^\lambda(z)}\big(|Du|+s\big)^p\dz\biggr)}^{\frac 1p}\\+M{\biggl(\mean{Q_r^\lambda(z)}\big(\Psi_{2R}+s\big)^{\eta}\dz\biggr)}^{\frac 1{\eta}}<\lambda. 
\end{multline}
Indeed, enlarging the domain of integration (notice that for cylinders as those we consider, we have $Q_r^\lambda(z)\subset Q_r(z)\subset Q_{2R}$ and hence $|Q_r^\lambda(z)|\leq|Q_{2R}|$) we infer, since $\lambda> B\lambda_0$
\begin{equation}\label{plugin}
CZ\big(Q_r^\lambda(z)\big)\leq 2\biggl[\frac{|Q_{2R}|}{|Q_r^\lambda(z)|}\biggr]^{\frac1p}\lambda_0^\frac 1d\leq 2\biggl[\frac{|Q_{2R}|}{|Q_r^\lambda(z)|}\biggr]^{\frac1p}B^{-\frac1d}\lambda^\frac 1d.
\end{equation}
Now in the case $p\geq2$ we estimate, recalling the definition of $d$
\[
\frac{|Q_{2R}|}{|Q_r^\lambda(z)|}\leq\lambda^{p-2}\Bigl(\frac{2R}{r}\Bigr)^N\leq \lambda^{p(1-\frac1d)}\Bigl(\frac{80R}{r_2-r_1}\Bigr)^N,
\]
while if $p<2$ we make the necessary changes, but we have the same result:
\[
\frac{|Q_{2R}|}{|Q_r^\lambda(z)|}\leq\lambda^\frac{(2-p)n}{2}\Bigl(\frac{2R}{r}\Bigr)^N\leq \lambda^{p(1-\frac1d)}\Bigl(\frac{80R}{r_2-r_1}\Bigr)^N.
\]
Hence, plugging these two estimates into \eqref{plugin}, depending clearly on the value of $p$, one immediately sees that \eqref{sabove} holds. On the other hand, if we consider points
\begin{multline*}
\bar z\in E(\lambda,Q_{r_1})\\:=\big\{z\in Q_{r_1}: \text{$z$ is a Lebesgue's point of $Du$ and }|Du(z)|+s>\lambda\big\}, 
\end{multline*}
for $\lambda>0$, by Lebesgue differentiation Theorem we have
\begin{equation}\label{bbelow}
\lim_{r\searrow0}CZ\big(Q_r^\lambda(\bar z)\big)\geq\lim_{r\searrow0}\biggl(\mean{Q_r^\lambda(\bar z)}\big(|Du|+s\big)^p\dz\biggr)^{\frac1p}=|Du(\bar z)|+s>\lambda 
\end{equation}
and therefore the converse inequality holds true. Hence, taking the previous two facts \eqref{sabove} and \eqref{bbelow} into account, we get from the absolute continuity of the integral that for each $\lambda> B\lambda_0$ and for every $\bar z\in E(\lambda,Q_{r_1})$ there exists a maximal radius $r_{\bar z}$ such that 
\begin{equation}\label{exit.cylinder}
{\biggl(\mean{Q_{r_{\bar z}}^\lambda(\bar z)}\big(|Du|+s\big)^p\dz\biggr)}^{\frac 1p}+M{\biggl(\mean{Q_{r_{\bar z}}^\lambda(\bar z)}(\Psi_{2R}+s)^{\eta}\dz\biggr)}^{\frac 1{\eta}}=\lambda; 
\end{equation}
we use the word maximal in the sense that for any $r\in(r_{\bar z},r_2-r_1]$,  $CZ(Q^\lambda_r(\bar z))<\lambda$. Note that by \eqref{sabove} we have $r_{\bar z}<(r_2-r_1)/40$ and therefore $Q^\lambda_{40r_{\bar z}}(\bar z)\subset Q_{r_2}$ since in particular $\bar z\in Q_{r_1}$. Moreover, we have 
\begin{multline}\label{intrinsic.big}
\frac{\lambda}{40^{\frac Np}}\leq \biggl(\mean{Q_{40r_{\bar z}}^\lambda(\bar z)}\big(|Du|+s\big)^p\dz\biggr)^{\frac1p}\\
+M\biggl(\mean{Q_{40r_{\bar z}}^\lambda(\bar z)}(\Psi_{2R}+s)^{\eta}\dz\biggr)^{\frac1{\eta}}\leq \lambda
\end{multline}
the left-hand side inequality reducing the integration domain to $Q_{r_{\bar z}}^\lambda(\bar z)$, the right-hand side from the aforementioned maximality of the radius $r_{\bar z}$.

\vs

We stress again that for the remainder of the proof, when dealing with cylinders of the type $Q_R^\lambda$ we shall implicitly understand which kind of parabolic cylinders we are using, depending on the value of $p$.

\subsection{A density estimate}

Fix here $\lambda>B\lambda_0$ and single out one of the cylinders previously chosen, say $Q\equiv Q^\lambda_{r_{\bar z}}(\bar z)$, such that $CZ(Q)=\lambda$. We must be in one of the following two cases:
\begin{equation}\label{split}
\Bigl(\frac{\lambda}{2}\Bigr)^p\leq \mean{Q}\big(|Du|+s\big)^p\dz\quad\text{or}\quad\Bigl(\frac{\lambda}{2}\Bigr)^{\eta}\leq M^{\eta}\mean{Q}(\Psi_{2R}+s)^{\eta}\dz. 
\end{equation}
In the case the first alternative holds, we split the average in the following way:
\begin{align}\label{part.fist.split}
&\Bigl(\frac{\lambda}{2}\Bigr)^p\leq \mean{Q}\big(|Du|+s\big)^p\dz\\
&\le \frac{|Q\smallsetminus E(\lambda/4,Q_{r_2})|}{|Q|}\Big(\frac\lambda4\Big)^p+\frac{1}{|Q|}\int_{Q\cap E(\lambda/4,Q_{r_2})}\!\!\big(|Du|+s\big)^p\dz\notag\\
&\leq \Big(\frac\lambda4\Big)^{p}+\biggl(\frac{|Q\cap E(\lambda/4,Q_{r_2})|}{|Q|}\biggr)^{\frac{\varepsilon_1}{1+\varepsilon_1}}\biggl(\mean{Q}\big(|Du|+s\big)^{p(1+\varepsilon_1)}\dz\biggr)^{\frac{1}{1+\varepsilon_1}},\notag
\end{align}
being $\varepsilon_1$ the higher integrability exponent of Theorem \ref{hiint.BS}. Thus, taking into account \eqref{exit.cylinder}, we have a constant depending on $n,p,\nu,L,\gamma$ {\em but not on $M$} such that
\[
\mean{Q}\big(|Du|+s\big)^{p(1+\varepsilon_1)}\dx\leq c\,\lambda^{p(1+\varepsilon_1)}.
\]
Therefore plugging this estimate in \eqref{part.fist.split}, reabsorbing $(\lambda/4)^p$, dividing by $\lambda^p$ and recalling that $Q=Q^\lambda_{r_{\bar z}}(\bar z)$ yields
\begin{equation}\label{meas.1}
 |Q^\lambda_{r_{\bar z}}(\bar z)|\leq c\,\big|Q^\lambda_{r_{\bar z}}(\bar z)\cap E(\lambda/4,Q_{r_2})\big|,
\end{equation}
with the constant depending on $n,p,\nu,L,\gamma$.

\vs

If, on the other hand, \eqref{split}$_2$ holds, take
\begin{equation}\label{choices}
\varsigma=\frac{1}{4M};
\end{equation}
then using Fubini's Theorem and splitting the integral
\begin{align*}
\Bigl(\frac{\lambda}{2M}\Bigr)^{\eta}&\leq\mean{Q}(\Psi_{2R}+s)^{\eta}\dz\\
&=\frac{{\eta}}{|Q|}\int_0^\infty\mu^{\eta}\big|\{z\in Q:\Psi_{2R}(z)+s>\mu \}\big|\frac{d\mu}{\mu}\\
&\leq (\varsigma\lambda)^{\eta}+\frac{{\eta}}{|Q|}\int_{\varsigma\lambda}^\infty\mu^{\eta}\big|\{z\in Q:\Psi_{2R}(z)+s>\mu \}\big|\frac{d\mu}{\mu}.
 \end{align*}
The choice of $\varsigma$ allows to reabsorb the first term on the left-hand side and to infer, dividing by $\lambda^{\eta}$ and recalling the expression for $\varsigma$
\[
|Q|\leq \frac{{\eta}}{(\varsigma\lambda)^{\eta}}\int_{\varsigma\lambda}^\infty\mu^{\eta}\big|\{z\in Q:\Psi_{2R}(z)+s>\mu \}\big|\frac{d\mu}{\mu}.
\]
Merging the estimate in the last display with \eqref{meas.1} gives
\begin{multline}\label{discuss}
 |Q^\lambda_{r_{\bar z}}(\bar z)|\leq c\, \big|Q^\lambda_{r_{\bar z}}(\bar z)\cap E(\lambda/4,Q_{r_2})\big|\\+ \frac{c}{(\varsigma\lambda)^{\eta}}\int_{\varsigma\lambda}^\infty\mu^{\eta}\big|\{z\in Q^\lambda_{r_{\bar z}}(\bar z):\Psi_{2R}(z)+s>\mu \}\big|\frac{d\mu}{\mu},
\end{multline}
with $c$ depending on $n,p,\nu,L,\gamma$ but not on $M$.

\subsection{Comparisons} \label{comparion.par}
We start with the solution $u\in K_0$ to \eqref{inequality} and a cylinder $40Q\equiv Q_{40r_{\bar z}}^\lambda({\bar z})$, ${\bar z}=({\bar x},{\bar t})$, defined as above, for $\lambda> B\lambda_0$ and ${\bar z}\in E(\lambda,Q_{r_1})$; hence we have $40Q\subset Q_{r_2}$ and that \eqref{intrinsic.big} holds.

\vs

{\em First comparison.} We want to build an admissible comparison function $v\in K_0'$ to be used in the variational inequality \eqref{inequality}, and to do this we shall solve an appropriate Cauchy-Dirichlet parabolic problem. We shall write $Q=40I\times 40B$ independently of the value of $p$; therefore for the meaning of $40I$ and $40B$ we refer to \eqref{cylinders}. Take the solution
\[
v\in u+L^p(40I;W^{1,p}_0(40B))\cap C^0(\overline{40I};L^2(40B))
\]
to
\begin{equation}\label{f.comparison}
 \begin{cases}
 \partial_tv-\divergence a(x,t,Dv)= \partial_t\psi-\divergence a(x,t,D\psi)&\text{in $40Q$,}\\[3pt]
v= u&\text{on $\partial_\Ph(40Q)$,}
\end{cases}
\end{equation}
where $\psi\in $ is the obstacle; existence of such a function is a classic fact since the right-hand side belongs to $L^{p'}(40Q)$ by \eqref{int.ostacolo}-\eqref{reg.ostacolo} and the boundary value $u$ belongs to the energy space; moreover we clearly have, by difference and \eqref{int.ostacolo}-\eqref{reg.ostacolo}, that
\[
\partial_tv\in L^{p'}(40I;W^{-1,p'}(40B))
\]
in the following sense: for $\varphi \in W^{1,p}_0(40B)$ and for a.e. $t\in 40I$
\begin{multline*}
\langle\partial_tv(\cdot,t),\varphi\rangle_{W^{-1,p'}\times W^{1,p}_0}\\=\int_{40B}\langle a(x,t,D\psi)-a(x,t,Dv),D\varphi\rangle\dx+\int_{40B}\partial_t\psi\varphi\dx;
\end{multline*}
 moreover the map $t\mapsto\langle\partial_tv(\cdot,t),\varphi\rangle_{W^{-1,p'}\times W^{1,p}_0}$ belongs to $ L^{p'}(40I)$. By comparison Lemma \ref{comparison} we infer that $v\geq\psi$ on $40Q$, since $v=u\geq\psi$ on $\partial_\Ph(40Q)$. If we now extend $v$ to $\Omega_{\bar t}$ (keeping denoted it by $v$) by setting $v=u$ in $\Omega_{\bar t}\smallsetminus 40Q$, this gives an admissible test function for the localized inequality \eqref{inequality.localized}, so that we have, after changing sign
\begin{multline}\label{partic}
-\int_{40I}\langle\partial_t v,v-u\rangle_{W^{-1,p}\times W^{1,p}_0}\dt-\int_{40Q}\langle a(x,t,Du),Dv-Du\rangle\dz\\\leq-\int_{40Q}\langle |F|^{p-2}F,Dv-Du\rangle\dz-\int_{40Q}f(v-u)\dz,
\end{multline}
taking into account that the extension of $v$ agrees with $u$ outside $40Q$ and hence also the term $\int_{40B}|v(\cdot,0)-u_0|^2\dx$ disappears.
On the other hand, using as test function $\varphi=v-u\in W^{1,p}_0(40B)$ in the weak formulation of \eqref{f.comparison}$_1$ and adding it to \eqref{partic} we get
\begin{multline*}
\mean{40Q}\langle a(x,t,Dv)-a(x,t,Du),Dv-Du\rangle\dz\\\leq\mean{40Q}\langle a(x,t,D\psi)-|F|^{p-2}F,Dv-Du\rangle\dz+\mean{40Q}(\partial_t\psi-f)(v-u)\dz,
\end{multline*}
after taking averages. Using now \eqref{monotonicity.V} to estimate the left-hand side from below we deduce
\begin{multline}\label{f.estimate.V}
\mean{40Q}\big|V_s(Dv)-V_s(Du)\big|^2\dz\\\leq c\mean{40Q}\langle a(x,t,D\psi)-|F|^{p-2}F,Dv-Du\rangle\dz\\+c\mean{40Q}(\partial_t\psi-f)(v-u)\dz
\end{multline}
with $c\equiv c(n,p,\nu)$. At this point we shall use the previous inequality in two directions. First we use Lemma \ref{fst.lemma} to get the following bound for the energy of $Dv$:
\begin{align}\label{ex.energy}
&\mean{40Q}\big(|Dv|+s\big)^p\dz\leq c\mean{40Q}\big(|Du|+s\big)^p\dz\notag\\
&+ c\mean{40Q}\langle a(x,t,D\psi)-|F|^{p-2}F,Dv-Du\rangle\dz\notag\\
&\hspace{5cm}+c\mean{40Q}(\partial_t\psi-f)(v-u)\dz,
\end{align}
for a constant depending on $n,p,\nu$; we call the three terms on the right-hand side $c\,I_1,c\,II_1$ and $c\,III_1$. We simply estimate $I_1\leq c\,\lambda^p$ by \eqref{intrinsic.big}; also the estimate for the remaining terms are easy. Indeed using the growth condition \eqref{assumptions}$_2$ and Young's inequality with $\varepsilon\in(0,1)$ to be chosen, we have
\begin{align}\label{anc.recall}
II_1&\leq c\mean{40Q}\Bigl[ \big(s^2+|D\psi|^2\big)^{\frac{p-1}{2}}+|F|^{p-1}\Bigr]\big(|Du|+|Dv|\big)\dz\\
&\leq c\,I_1+c\,\varepsilon^{-p'}\mean{40Q}\Bigl[ |D\psi|^p+|F|^p+s^p\Bigr]\dz+\tilde c\,\varepsilon^p\mean{40Q}|Dv|^p\dz\notag
\end{align}
$\tilde c$ depending on $n,p,\nu,L$. In the same way, using Poincar\'e's inequality slicewise (note that $(v-u)(\cdot,t)\in W^{1,p}_0(40B)$ for a.e. $t\in 40I$) we estimate, again with Young's inequality, for $\varepsilon$ as above and triangle inequality
\begin{multline}\label{anc.recall.2}
III_1\leq  c\,\varepsilon^{-p'}r_{\bar z}^{p'}\mean{40Q}\big({|\partial_t\psi|}^{p'}+{|f|}^{p'})\dz\\+\bar c\,\varepsilon^p\mean{40Q}\big({|Du|}^p+{|Dv|}^p\big)\dz .
\end{multline}
At this point we choose $\varepsilon$ small enough, so that $\varepsilon^p(\tilde c+\bar c)=1/2$; reabsorbing the right-hand side energy of $Dv$, recalling that $r_{\bar z}\leq2R$ and the definition of $\Psi_{2R}$ in \eqref{Psi} and also \eqref{intrinsic.big} and performing simple algebraic manipulation, gives
\begin{equation}\label{energy.Dv}
\mean{40Q}\big(|Dv|+s\big)^p\dz\leq c\,\lambda^p+c\mean{40Q}\Psi_{2R}^p\dz\leq \big[c+\frac cM\big]\lambda^p\leq c\lambda^p
\end{equation}
since $M\geq1$, for a constant depending only on $n,p,\nu,L$. The second estimate we shall deduce from \eqref{f.estimate.V} is the following comparison one: up to a constant, we still have the two terms $II_1,III_1$ on the right-hand side; that is
\[
\mean{40Q}\big|V_s(Du)-V_s(Dv)\big|^2\dz\leq \frac{1}{c_\ell}[II_1+III_1];
\]
here we simply use Young's inequality in a different way (we don't have to reabsorb the energy of $Dv$). We have, recalling \eqref{anc.recall}, using H\"older's inequality, \eqref{intrinsic.big} and \eqref{energy.Dv}
\begin{align*}
II_1&\leq \frac{c(p)}{M^{p-1}}\biggl(M^p\mean{40Q}\big(|D\psi|^p+|F|^p+s^p\big)\dz\biggr)^{\frac{1}{p'}}\times\\
&\hspace{5cm}\times\biggl(\mean{40Q}\big(|Du|+|Dv|\big)^p\dz\biggr)^\frac1p \\
&\leq\frac{c}{M^{p-1}}\lambda^{\frac{p}{p'}+1}=\frac{c}{M^{p-1}}\lambda^p.
\end{align*}
Similarly we can bound $III_1$ by
\begin{multline*}
\frac{c}{M^{p-1}}\biggl(M^pr_{\bar z}^{p'}\mean{40Q}\big(|\partial_t\psi|+|f|\big)^{p'}\dz\biggr)^{\frac{1}{p'}}\biggl(\mean{40Q}|Dv-Du|^p\dz\biggr)^{\frac1p}\\
\leq \frac{c}{M^{p-1}}\lambda^p; 
\end{multline*}
hence, all in all we have
\begin{equation}\label{f.compa.V}
\mean{40Q}\big|V_s(Du)-V_s(Dv)\big|^2\dz\leq \frac{c}{M^{p-1}}\lambda^p
\end{equation}
for a constant $c$ depending only on $n,p,\nu,L$; this is the first comparison estimate we were looking for.

\vs

{\em Second comparison.} 
On a smaller cylinder now we want to consider the function agreeing with $v$ on the parabolic boundary but solving an homogeneous parabolic equation. Therefore, with the same notation introduced in the previous step, we consider the solution 
\[
\tilde v\in u+L^p(20I;W^{1,p}_0(20B))\cap C^0(\overline{20I};L^2(20B))
\]
to the Cauchy-Dirichlet problem
\begin{equation}\label{s.comparison}
 \begin{cases}
 \partial_t\tilde v-\divergence a(x,t,D\tilde v)= 0&\text{in $20Q$,}\\[3pt]
\tilde v= v&\text{on $\partial_\Ph(20Q)$;}
\end{cases}
\end{equation}
also here the existence is guaranteed by classic results. We test now the weak formulation of \eqref{f.comparison} with $v-\tilde v$ as test function, extended to zero in $40Q\smallsetminus20Q$ and that of \eqref{s.comparison}, tested with $v-\tilde v$; notice that in both cases a regularization in time via Steklov averaging is needed; however we shall proceed formally, here, subtracting the second from the first one (we could also follow \cite[Lemma 2.1]{BDM2}, in a more abstract setting). We hence have
\begin{multline*}
\mean{20Q}\partial_t(v-\tilde v)(v-\tilde v)\dz+\mean{20Q}\langle a(x,t,Dv)-a(x,t,D\tilde v),Dv-D\tilde v\rangle\dz\\=\mean{20Q}\big[\langle a(x,t,D\psi),Dv-D\tilde v\rangle+\partial_t\psi(v-\tilde v)\big]\dz.
\end{multline*}
We call $I_2$ and $II_2$ the terms on the left-hand side (respectively, the parabolic and the elliptic one) and $III_2$ the term on the right-hand side. For the parabolic term we have
\[
I_2=\frac12\mean{20Q}\partial_t|v-\tilde v|^2\dz=\frac12\mean{20B}|v-\tilde v|^2(\cdot,\bar t)\dx\geq0;
\]
therefore we can discard it. Monotonicity formula \eqref{monotonicity.V} tells that we can bound
\[
II_2\geq \frac{1}{c_m(n,p,\nu)}\mean{20Q}\big|V_s(Dv)-V_s(D\tilde v)\big|^2\dz,
\]
while for the remaining one we estimate, similarly as in \eqref{anc.recall}-\eqref{anc.recall.2}, using Young's inequality twice, Poincar\'e's inequality and enlarging the domain of integration:
\begin{align*}
III_2&\leq  c(n,L)\mean{20Q}\biggl[ {\big(s^2+|D\psi|^2\big)}^{\frac{p-1}{2}}\big(|Dv|+|D\tilde v|\big)\\
&\hspace{4.5cm}+\varepsilon^{-p'}r_{\bar z}^{p'}|\partial_t\psi|^{p'}+\varepsilon^p\Bigl(\frac{|v-\tilde v|}{20r_{\bar z}}\Bigr)^p\biggr]\dz\\
&\leq  \varepsilon c(n,L)\mean{20Q}|D\tilde v|^p\dz+c(n,p,L)\mean{40Q} {\big(|Dv|+s\big)}^p\dz\\
&\hspace{2cm}+c(n,p,L,\varepsilon)\mean{20Q}\Bigl[|D\psi|^p+(2R)^{p'}|\partial_t\psi|^{p'}+s^p\Bigr]\dz\\
&\leq  \frac{1}{2c_mc_\ell}\mean{20Q}|D\tilde v|^p\dz+c(n,p,\nu,L)\,\lambda^p
\end{align*}
choosing $\varepsilon\equiv \varepsilon(n,p,\nu,L)$ small enough in the last line, taking into account \eqref{energy.Dv} and \eqref{intrinsic.big}, exactly as done to obtain \eqref{energy.Dv}. Hence appealing to Lemma \ref{fst.lemma} and again to \eqref{energy.Dv}, we have the following energy estimate for $D\tilde v$:
\begin{equation}\label{energy.Dtildev}
 \mean{20Q}\big(|D\tilde v|+s\big)^p\dz\leq c\,\lambda^p
\end{equation}
for a constant again depending on $n,p,\nu,L$. Now, working in a competely analogous way as done to deduce \eqref{f.compa.V}, estimating quite differently $III_2$, one can deduce the following
\begin{equation}\label{s.compa.V}
\mean{20Q}\big|V_s(Dv)-V_s(D\tilde v)\big|^2\dz\leq \frac{c}{M^{{p-1}}}\lambda^p;
\end{equation}
the (easy) details are left to the reader. Before passing to the next step, however, let us stress that due to \eqref{energy.Dtildev} we have that the higher-integrability estimate \eqref{hi.int.homo} reads here as follows:
\begin{equation}\label{hi.estimate.Dtildev}
 \mean{10Q}\big(|D\tilde v|+s\big)^{p(1+\varepsilon_0)}\dz\leq c\,\lambda^{p(1+\varepsilon_0)}.
\end{equation}

\vs

{\em Third comparison.} 
We finally come to the last comparison, and here we want to get rid of the $(x,t)$ dependence using \eqref{mod.osc} and comparing $\tilde v$ to another appropriate more regular solution $w$, having bounded gradient. We consider here the cylinder $10Q$ and for shortness we denote the averaged vector field
\[
\tilde a(t,\xi):=(a)_{10B}(t,\xi)=\mean{10B}a(\cdot,t,\xi)\dx,
\]
for any $t\in 10I$ and for all $\xi\in\R^n$, accordingly with \eqref{av.vf}. Now we define the solution
\[
w\in v+ L^p(10I;W^{1,p}_0(10B))\cap C^0(\overline{10I};L^2(10B))
\]
to the initial-lateral boundary value problem
\[
\begin{cases}
          \partial_t w -\divergence\tilde a(t,Dw)=0\qquad&\text{in $10Q$},\\[0.3cm]
          w=\tilde v&\text{on $\partial_\Ph (10Q)$}.
\end{cases}
\]
The usual procedure, already applied, after discarding the positive term gives
\begin{align}\label{start.Dw}
I_3+II_3&:=\mean{10Q}\langle \tilde a(t,D\tilde v)-\tilde a(t,Dw),D\tilde v-Dw\rangle\dz\\
&\quad+\mean{10Q}\langle a(x,t,D\tilde v)-\tilde a(t,D\tilde v),D\tilde v-Dw\rangle\dz \notag\\
&\quad\quad= \mean{10Q}\langle a(x,t,D\tilde v)-\tilde a(t,Dw),D\tilde v-Dw\rangle\dz\leq0.\notag
\end{align}
{\em Energy estimate for $Dw$.} We split, after using the growth condition \eqref{assumptions}$_2$
\begin{multline*}
-II_3\leq 2L\mean{10Q} {\big(s+|D\tilde v|\big)}^{p-1}|D\tilde v|\dz\\+2L\mean{10Q}{\big(s+|D\tilde v|\big)}^{p-1}|Dw|\dz; 
\end{multline*}
we estimate in the first term
\[
\mean{10Q} {\big(s+|D\tilde v|\big)}^{p-1}|D\tilde v|\dz\leq c(n)\mean{20Q} {\big(s+|D\tilde v|\big)}^p\dz\leq c\,\lambda^p,
\]
$c\equiv c(n,p,\nu,L)$, while for the second we estimate, using Young's inequality, $\varepsilon \in (0,1)$ to be fixed and \eqref{energy.Dtildev}
\begin{align}\label{reab.Dw}
\mean{10Q} {\big(s+|D\tilde v|\big)}^{p-1}&|Dw|\dz\notag\\
&\leq \varepsilon \mean{10Q}|Dw|^p\dz+c(p,\varepsilon)\mean{10Q} {\big(s+|D\tilde v|\big)}^p\dz\notag\\
&\leq  \varepsilon \mean{10Q}|Dw|^p\dz+c(n,p,\nu,L,\varepsilon)\lambda^p.
\end{align}
Therefore, first estimating $I_3$ from below with \eqref{monotonicity.V}, then using Lemma \ref{fst.lemma} as done in \eqref{ex.energy} (notice that \eqref{monotonicity.V} also, clearly, apply to $\tilde a(t,\xi)$) and finally re-absorbing the energy of $Dw$ appearing in \eqref{reab.Dw} we get
\begin{equation}\label{energy.Dw}
\mean{10Q}\big(s+|Dw|\big)^p\dz\leq c\,\lambda^p 
\end{equation}
for $c\equiv c(n,p,\nu,L)$. To complete this list of estimate, we come to the third comparison one; again, we want to take into account a smallness condition, that in this case will be  given by \eqref{mod.osc}. We start again from \eqref{start.Dw}. 

\vs

\noindent {\em Comparison estimate for $Dw$.}  Here we have to go trough a different path, since we need to encode a smallness condition in the estimate; on the other hand, we can use the just proved energy estimate \eqref{energy.Dw}. We call
\[
A_s(D\tilde v,10B):=\frac{|a(x,t,D\tilde v)-\tilde a(t,D\tilde v)|}{{(s+|D\tilde v|)}^{p-1}}.
\]
We now have
\begin{align*}
\mean{10Q}\big|V_s(D\tilde v)-&V_s(Dw)\big|^2\dz\\&\leq  \mean{10Q}\langle\tilde a(t,D\tilde v)-\tilde a(t,Dw),D\tilde v-Dw\rangle\dz\\
&\leq  \mean{10Q}\langle\tilde a(t,D\tilde v)- a(x,t,D\tilde v),D\tilde v-Dw\rangle\dz
\end{align*}
by \eqref{monotonicity.V} and \eqref{start.Dw}. Now we estimate, using H\"older's inequality with exponents $p,p'(1+\varepsilon_0),p'(1+\varepsilon_0)'$, where $\varepsilon_0$ is the exponent in \eqref{hi.estimate.Dtildev},
\begin{align*}
&\mean{10Q}\langle\tilde a(t,D\tilde v)- a(x,t,D\tilde v),D\tilde v-Dw\rangle\dz\\
&\leq \mean{10Q}A_s(D\tilde v,10B) {\big(s+|D\tilde v|\big)}^{p-1}|Dw-D\tilde v|\dz\\ 
&\leq \biggl(\mean{10Q}\big[|Dw|^p+|D\tilde v|^p\big]\dz\biggr)^\frac 1p\biggl(\mean{10Q} {(s+|D\tilde v|)}^{p(1+\varepsilon_0)}\dz\biggr)^{\frac{1}{p'(1+\varepsilon_0)}}\times\\
&\hspace{3.5cm}\times\biggl(\mean{10Q}\big[A_s(D\tilde v,10B) \big]^{p'(1+\varepsilon_0)'}\dz\biggr)^{\frac{\varepsilon_0}{p'(1+\varepsilon_0)}}.
\end{align*}
Now we bound the first averaged integral using \eqref{energy.Dtildev} and \eqref{energy.Dw}, the second with \eqref{hi.estimate.Dtildev}; hence, taking finally into account that $A_s(D\tilde v,10B)\leq2L$ and H\"older's inequality
\begin{align}\label{compa.Dw}
\mean{10Q}&\big|V_s(D\tilde v)-V_s(Dw)\big|^2\dz\notag\\
&\leq c\,\lambda^{1+\frac{p}{p'}}\biggl(\mean{10I}\Bigl(\mean{10B}\Bigl[\frac{|a(x,t,D\tilde v)-\tilde a(t,D\tilde v)|}{{(s+|D\tilde v|)}^{p-1}}\Bigr]^2\dx\Bigr)^{\frac12}\biggr)^{\frac{\varepsilon_0}{1+\varepsilon_0}}\notag\\
&\leq c\,\lambda^p\big[\omega_a(10r_z)\big]^{\bar\varepsilon};
\end{align}
$\bar\varepsilon\in(0,1)$ is an exponent depending on $\varepsilon_0$ and hence on $n,p,\nu,L$.

\subsection{Level-set estimates}\label{level.set}
Take a point $\bar z\in E(A\lambda,Q_{r_1})$, for $A\geq1$ to be chosen; hence $|Du(\bar z)|+s>A\lambda$ and in particular $\bar z\in E(\lambda,Q_{r_1})$. Therefore we can consider the cylinder $Q_{r_{\bar z}}^\lambda(\bar z)$ previously defined, where \eqref{exit.cylinder} and \eqref{intrinsic.big} hold. Define the comparison functions $v$, $\tilde v$ and $w$, respectively, over the cylinders $Q^\lambda_{40r_{\bar z}}(\bar z)$, $Q^\lambda_{20r_{\bar z}}(\bar z)$ and $Q^\lambda_{10r_{\bar z}}(\bar z)$, as in Paragraph \ref{comparion.par}. 

\vs

Observe now that $w$ is solution to a systems with just time-dependent coefficients and therefore $Dw$ turns out to be locally bounded in $Q^\lambda_{10r_{\bar z}}(\bar z)$, see Theorem \ref{sup.Thm}, and by estimate \eqref{energy.Dw} we have that
\begin{equation}\label{DiBenedetto}
	\sup_{Q^\lambda_{5r_{\bar z}}(\bar z)}|Dw|+s\leq c_{DiB}\lambda, 
\end{equation}
with $c_{DiB}$ just depending on $n,p,\nu,L$ but not on the cylinder, neither on $\lambda$. We shall use this to prove that
\begin{multline}\label{cases}
	\big(|Dw(z)|+s\big)^p \leq\big|V_s(D\tilde v(z))-V_s(Dw(z))\big|^2\\+\big|V_s(Dv(z))-V_s(D\tilde v(z))\big|^2+\big|V_s(Du(z))-V_s(Dv(z))\big|^2
\end{multline}
holds for any $z\in Q^\lambda_{5r_{\bar z}}(\bar z)\cap E(A\lambda,Q_{r_2})$, for an appropriate choice of $A$. Indeed using Lemma \ref{fst.lemma} three times, we infer the inequality
\begin{multline}\label{split2}
	\big(|Du(z)|+s\big)^p\leq c_\ell^3\,\big(s+|Dw(z)|\big)^p+c_\ell^3\,\big|V_s(D\tilde v(z))-V_s(Dw(z))\big|^2\\
	+c_\ell^2\,\big|V_s(Dv(z))-V_s(D\tilde v(z))\big|^2+c_\ell\,\big|V_s(Du(z))-V_s(Dv(z))\big|^2.
\end{multline}
Suppose now that \eqref{cases} fails: this, together with the latter inequality would yield
\[
\big(|Du(z)|+s\big)^p< 2c_\ell^3\big(s+|Dw(z)|\big)^p
\]
and then, also by \eqref{DiBenedetto} and the fact that $|Du(z)|+s>A\lambda$
\begin{align*}
\big(|Dw(z)|+s\big)^p\le c_{DiB}^p\,\lambda^p&<c_{DiB}^p\frac{(|Du(z)|+s)^p}{A^p}\\&< \frac{2c_\ell^3c_{DiB}^p}{A^p}\big(|Dw(z)|+s\big)^p,
\end{align*}
which is a contradiction for the choice 
\[
A\equiv A(n,p,\nu,L):=(2c_\ell^3)^{\frac1p}c_{DiB}\geq1. 
\]
Combining \eqref{cases} and \eqref{split2} we thus get
\begin{multline*}
	\big(|Du(z)|+s\big)^p\leq 2c_\ell^3\Bigl[\big|V_s(D\tilde v(z))-V_s(Dw(z))\big|^2\\
	+\big|V_s(Dv(z))-V_s(D\tilde v(z))\big|^2+\big|V_s(Du(z))-V_s(Dv(z))\big|^2\Bigr]
\end{multline*}
for all $z\in Q^\lambda_{5r_{\bar z}}(\bar z)\cap E(A\lambda,Q_{2R})$. Hence
\begin{align*}
&\big|\{z\in Q^\lambda_{5r_{\bar z}}(\bar z):\big(|Du(z)|+s\big)^p>A\lambda\}\big|\notag\\
	&\leq\big|\{z\in Q^\lambda_{5r_{\bar z}}(\bar z):\big|V_s(Du(z))-V_s(Dv(z))\big|^2>\frac{{(A\lambda)}^p}{8c_\ell^3}\}\big|\notag\\
	&\qquad+\big|\{z\in Q^\lambda_{5r_{\bar z}}(\bar z):\big|V_s(Dv(z))-V_s(D\tilde v(z))\big|^2>\frac{{(A\lambda)}^p}{8c_\ell^3}\}\big|\notag\\
	&\qquad\qquad+\big|\{z\in Q^\lambda_{5r_{\bar z}}(\bar z):|Dw(z)-D\tilde v(z)|^p>\frac{{(A\lambda)}^p}{8c_\ell^3}\}\big|.
\end{align*}
Now recalling that $A$ is fixed as a constant depending on $n,p,\nu,L$ and enlarging appropriately the domains of integration we get, using also the comparison estimates \eqref{f.compa.V}, \eqref{s.compa.V} and \eqref{compa.Dw} and finally the density estimate \eqref{discuss}
\begin{align}\label{int.estimate}
&\big|\{z\in Q^\lambda_{5r_{\bar z}}(\bar z):\big(|Du(z)|+s\big)^p>A\lambda\}\big|\notag\\
	&\leq \frac{c}{\lambda^p}\int_{Q^\lambda_{40r_{\bar z}}(\bar z)}\big|V_s(Du)-V_s(Dv)\big|^2 \dz\notag\\
	&\qquad+ \frac{c}{\lambda^p}\int_{Q^\lambda_{20r_{\bar z}}(\bar z)}\big|V_s(Dv)-V_s(D\tilde v)\big|^2 \dz\notag\\
	&\qquad\qquad+\frac{c}{\lambda^p}\int_{Q^\lambda_{10r_{\bar z}}(\bar z)}\big|V_s(D\tilde v)-V_s(Dw)\big|^2 \dz\notag\\
	&\qquad\leq c\, \Big[\frac{1}{M^{p-1}} +\big[\omega_a(10r_{\bar z})\big]^{\bar\varepsilon}\Big]\big|Q^\lambda_{r_{\bar z}}(\bar z)\big|\notag\\
	&\qquad\leq c\,G(2R,M)\biggl[\big|Q^\lambda_{r_{\bar z}}(\bar z)\cap E(\lambda/4,Q_{r_2})\big|\notag\\
	&\qquad+\frac{1}{(\varsigma\lambda)^{\eta}}\int_{\varsigma\lambda}^\infty\mu^{\eta}\big|\{z\in Q^\lambda_{r_{\bar z}}(\bar z):\Psi_{2R}(z)+s>\mu \}\big|\frac{d\mu}{\mu}\bigg],
\end{align}
where we denoted by $G(2R,M)$ the quantity $M^{1-p} +[\omega_a(2R)]^{\bar\varepsilon}$; the constant $c$ depends only on $n,p,\nu,L$. Note that we used clearly the monotonicity of $\rho\mapsto\omega_a(\rho)$ and the fact $10r_{\bar z}\leq 2R$.

\vs

Now consider the collection $\mathcal E_\lambda$ of cylinders $Q^\lambda_{r_{\bar z}}(\bar z)$, when $\bar z$ varies in $E(A\lambda,Q_{r_1})$. By a Vitali-type argument, we extract a countable sub-collection $\mathcal{F}_\lambda\subset\mathcal{E}_\lambda$ such that the $5$-times enlarged cylinders cover almost all $E(A\lambda,Q_{r_1})$ in the sense that if we denote the cylinders of $\mathcal{F}_\lambda$ by $Q_i^0:=Q^\lambda_{r_{\bar z_i}}(\bar z_i)$, for $i\in \mathcal{I}_\lambda$, being possibily $\mathcal{I}_\lambda=\N$, with their ``verteces'' $\bar z_i\in E(A\lambda,Q_R)$, we have
\[
E(A\lambda,Q_{r_1})\subset\bigcup_{i\in\mathcal{I}_\lambda} Q_i^1\cup\mathcal{N}_\lambda \qquad \text{with $|\mathcal{N}_\lambda|=0$}
\] 
and where we denoted $Q_i^1:=5Q_i^0=Q^\lambda_{5r_{\bar z_i}}(\bar z_i$. Moreover the cylinders are pairwise disjoints, i.e., $Q_i^0\cap Q_j^0=\emptyset$ whenever $i\neq j$. Using these two facts we can deduce an estimate for the measure of the level sets in the full $Q_{r_1}$ of $|Du|+s$: fix $\lambda>B\lambda_0$, take \eqref{int.estimate} over the cylinders of the covering $Q^\lambda_{5r_{\bar z}}=Q_i^1$ and sum over $\mathcal I_\lambda$: we get
\begin{multline}\label{est.first.par}
\big|E(A\lambda,Q_{r_1})\big|\leq c\,G(2R,M)\biggl[\big|E(\lambda/4,Q_{r_2})\big|\\
+\frac{1}{(\varsigma\lambda)^{\eta}}\int_{\varsigma\lambda}^\infty\mu^{\eta}\big|\{z\in Q_{r_2}:\Psi_{2R}(z)+s>\mu \}\big|\frac{d\mu}{\mu}\bigg].
\end{multline}
Finally we show how to refine the previous estimate in order to be allowed to reabsorb the Lorentz norm of $Du$ on the right-hand side. We define the truncations 
\begin{equation}\label{notation.k}
{|Du(z)|}_k:=\min\big\{|Du(z)|,k\big\}\qquad\text{for $z\in\Omega_T$ and $k\in\N\cap[B\lambda_0,\infty)$} 
\end{equation}
and note that by \eqref{est.first.par}, calling $E_k(\lambda,Q_\rho):=\{z\in Q_\rho:{|Du(z)|}_k+s>\lambda\}$, we have
\begin{multline}\label{notation.k.est}
\big|E_k(A\lambda,Q_{r_1})\big|\leq c\,G(2R,M)\biggl[\big| E_k(\lambda/4,Q_{r_2})\big|\\+\frac{1}{(\varsigma\lambda)^{\eta}}\int_{\varsigma\lambda}^\infty\mu^{\eta}\big|\{z\in Q_{r_2}:\Psi_{2R}(z)+s>\mu \}\big|\frac{d\mu}{\mu}\bigg],
\end{multline}
for $k\in\N\cap[B\lambda_0,\infty)$. Indeed in the case $k\le A\lambda$ we have $E_k(A\lambda,Q_{r_1})=\emptyset$ and therefore the previous estimate holds trivially. In the case $k>A\lambda$ on the other hand it follows since $E_k(A\lambda,Q_{r_1})= E(A\lambda,Q_{r_1})$ and
$ E_k(\lambda/4,Q_{2R}) =  E(\lambda/4,Q_{2R})$. 

\subsection{Conclusion, case $\boldsymbol{q<\infty}$}
Now the proof goes on exactly as in \cite{BL}, since the estimate we start from is very similar to \cite[Inequality $(5.14)$]{BL}; we sketch the details, referring to the aforementioned paper for more details. Multiply inequality \eqref{est.first.par} by $(A\lambda)^\gamma$ for $\gamma> p$, then raise both sides to the power $q/\gamma$ for $q<\infty$ and integrate with respect to the measure $d\lambda/(A\lambda)$ over $B\lambda_0$, since \eqref{est.first.par} holds true just for $\lambda$ varying in this range. This yields, recalling again that $A\geq1$ is a constant depending on $n,p,\nu,L$ and $\varsigma$ depends on $p,M$
\begin{align}\label{est.second.par}
&\int_{B\lambda_0}^\infty\Bigl((A\lambda)^\gamma\big|\{z\in Q_R:{|Du(z)|}_k+s>A\lambda\}\big|\Bigr)^{\frac q\gamma}\frac{d\lambda}{A\lambda}\notag\\
&\leq c\, [G(2R,M)]^{\frac q\gamma}\biggl[\int_0^\infty\Bigl(\lambda^\gamma\big|\{z\in Q_{r_2}:{|Du(z)|}_k+s>\lambda/4\}\big|\Bigr)^{\frac q\gamma}\frac{d\lambda}{\lambda}\notag\\
	&\ +c(p,\gamma,q,M)\times\notag\\&\qquad\times\int_0^\infty\lambda^{q(1-\frac {\eta}\gamma)}\biggl(\int_{\varsigma\lambda}^\infty\mu^{\eta}\big|\{z\in Q_{r_2}:\Psi_{2R}(z)+s>\mu \}\big|\frac{d\mu}{\mu}\biggr)^{\frac q\gamma}\frac{d\lambda}{\lambda}\bigg]\notag\\
	&=:c\, [G(2R,M)]^{\frac q\gamma}\big[I+II\big];
\end{align}
$c$ depends on $n,p,\nu,L,\gamma,q$. A  change of variable yields $I=c(q)\||Du|_k+s\|_{L(\gamma,q)(Q_{r_2})}^q$. For $II$ the situation is a bit more involved, and we have to consider separately two different cases. The first one is when $q\geq\gamma$; after changing again variable $\lambda\leftrightarrow\varsigma\lambda$, recalling the definition of $\varsigma$ in \eqref{choices}, and then we use Lemma \ref{Hardy} with $f(\mu)=\mu^{{\eta}-1}|\{z\in Q_{r_2}:\Psi_{2R}(z)+s>\mu \}|$, $\alpha=q/\gamma\geq1$ and $r=q(1-{\eta}/\gamma)>0$ to infer
\[
II\leq \frac{c}{(\gamma-p)^{q/\gamma}}\int_0^\infty\!\lambda^{q(1-\frac {\eta}\gamma)+{\eta}\frac q\gamma}\big|\{z\in Q_{r_2}:\Psi_{2R}(z)+s>\lambda \}\big|^{\frac q\gamma}\frac{d\lambda}{\lambda}
\]
and the latter integral is nothing else than $\|\Psi_{2R}+s\|_{L(\gamma,q)(Q_{r_2})}^q $; here $c\equiv c(p,\gamma,q,M)$ and note that \eqref{well.posed} is satisfied since $\Psi_{2R}\in L^{\eta}(Q_{r_2})$.

\vs

 In the case ${0<q<\gamma}$ we use Lemma \ref{Stein} with $r={\eta} q/\gamma$, $\alpha_1= 1< \gamma/q=\alpha_2$ and $h(\mu)={|\{z\in Q_{r_2}:\Psi_{2R}(z)+s>\mu\}|}^{\frac q\gamma}$:
\begin{multline*}
\biggl[\int_\lambda^\infty\mu^{\eta}\big|\{z\in Q_{r_2}:\Psi_{2R}(z)+s>\mu\}\big|\frac{d\mu}{\mu}\biggr]^{\frac q\gamma}\\
\leq  \lambda^{{\eta}\frac q\gamma } \big|\{z\in Q_{r_2}:\Psi_{2R}(z)+s>\lambda\}\big|^{\frac q\gamma}\\
+c\, \int_\lambda^\infty\mu^{{\eta}\frac q\gamma}\big|\{z\in Q_{r_2}:\Psi_{2R}(z)+s>\mu\}\big|^{\frac q\gamma}\frac{d\mu}{\mu}.
 \end{multline*}
Putting this estimate into the expression on the right-hand side of \eqref{est.second.par}, again after changing variable $\varsigma\lambda\leftrightarrow\lambda$
\begin{align*}
 II&\leq  c\,\|\Psi_{2R}+s\|_{L(\gamma,q)(Q_{r_2})}^q\\
&+c\,\int_0^\infty\lambda^{q(1-\frac {\eta}\gamma )}\biggl[\int_\lambda^\infty\mu^{{\eta}\frac q\gamma-1}\big|\{z\in Q_{r_2}:\Psi_{2R}(z)+s>\mu\}\big|^{\frac q\gamma}\,d\mu\biggr]\frac{d\lambda}{\lambda}\\
&\leq  \frac{c}{\gamma-p}\,\|\Psi_{2R}+s\|_{L(\gamma,q)(Q_{r_2})}^q,
\end{align*}
by Fubini's Theorem, $c\equiv c(p,\gamma,q,M)$. Therefore, all in all, putting all these estimates in \eqref{est.second.par}, after simple manipulations, we have that for all $\gamma>p$ and $0<q<\infty$, 
\begin{multline*}
\big\|{|Du|}_k+s\big\|_{L(\gamma,q)(Q_{r_1})}\leq c\,B\lambda_0|Q_R|^{\frac 1\gamma}\\+ \tilde c\, [G(2R,M)]^{\frac 1\gamma}\Bigl[\big\|{|Du|}_k+s\big\|_{L(\gamma,q)(Q_{r_2})}+c(M)\|\Psi_{2R}+s\|_{L(\gamma,q)(Q_{r_2})}\Bigr]
 \end{multline*}
with $\tilde c$ depending only on $n,p,\nu,L,\gamma,q$. At this point the reader might recall the definition of $G(2R,M)$ after \eqref{int.estimate}; we choose $M$ large enough and $R_0$ small enough to have
\[
\frac{\tilde c}{M^{p-1}}\leq \frac{1}{2^\gamma},\qquad\qquad\tilde c\big[\omega_a(2R_0)\big]^{\bar\varepsilon}\leq \frac{1}{2^\gamma}
\]
and this, taking into consideration the dependencies of $\tilde c$ and $\bar\varepsilon$, yields the dependencies for $R_0$ stated in Theorem \ref{thm.main}. Now we also have that $M$ is a fixed constant depending on $n,p,\nu,L,\gamma,q$; recall that $\bar\varepsilon$ has been defined after \eqref{compa.Dw}.

With these choices and taking into account that $d\geq1$, we have
\begin{multline}\label{pov.us}
\big\|{|Du|}_k+s\big\|_{L(\gamma,q)(Q_{r_1})}\leq \frac12\big\|{|Du|}_k+s\big\|_{L(\gamma,q)(Q_{r_2})}\\+c\|\Psi_{2R}+1\|_{L(\gamma,q)(Q_{r_2})}^d
+c\,|Q_{2R}|^{\frac 1\gamma}\Big(\frac{R}{r_2-r_1}\Big)^{\frac Npd}\times\\\times\biggl[{\biggl(\mean{Q_{2R}}\big(|Du|+s\big)^p\dz\biggr)}^{\frac dp}+{\biggl(\mean{Q_{2R}}(\Psi_{2R}+1)^{\eta}\dz\biggr)}^{\frac d{\eta}}\biggr]
 \end{multline}
all the constants depending on $n,p,\nu,L,\gamma,q$. At this point Lemma \ref{it.lemma} allows to re-absorb the $L(\gamma,q)$ norm of $Du$ in the right-hand side: $\|{|Du|}_k+s\|_{L(\gamma,q)(Q_{2R})}$ is clearly finite. Moreover first using the H\"older's inequality in Marcinkiewicz spaces Lemma \ref{lab.Marci}, then using \eqref{Bram} with $\lambda=0$ to get the Lorentz norm of $\Psi_{2R}+s$ from the Marcinkiewicz one (see \cite[$(5.19)$--$(5.20)$]{BL} for the missing details), we get
\[
 \biggl(\mean{Q_{2R}}\big(\Psi_{2R}+1\big)^{\eta}\dz\biggr)^{\frac 1{\eta}}\leq \frac{c(p,\gamma,q)}{(\gamma-p)^{1/{\eta}}}|Q_{2R}|^{-\frac 1\gamma}\|\Psi_{2R}+1\|_{L(\gamma,q)(Q_{2R})}.
\]
Putting all these informations in \eqref{pov.us} yields
\begin{multline*}
	\big\|{|Du|}_k+s\big\|_{L(\gamma,q)(Q_R)}\leq c\,|Q_R|^{\frac 1\gamma}\biggl[{\biggl(\mean{Q_{2R}}\big(|Du|+s\big)^p\dz\biggr)}^{\frac dp}\\+|Q_{2R}|^{-\frac d\gamma}{\|\Psi_{2R}+1\|}_{L(\gamma,q)(Q_{2R})}^d\biggr];
\end{multline*}
finally dividing by $|Q_R|^\frac1\gamma$, taking the limit $k\to\infty$ and using Fatou's Lemma together with Remark \ref{lsc} yields \eqref{main.est} for $q<\infty$.

\subsection{Conclusion, case  $\boldsymbol{q=\infty}$}
We have to come back to the second alternative in \eqref{split}. This time we split, for $\tau$ small to be chosen
\begin{align*}
\Bigl(\frac{\lambda}{2}\Bigr)^{\eta}&\leq M^{\eta}\mean{Q}(\Psi_{2R}+s)^{\eta}\dz\\&\leq M^{{\eta}}(\tau\lambda)^{\eta}+\frac{M^{{\eta}}}{|Q|}\int_{\Psi_{2R}(\tau\lambda,Q)}(\Psi_{2R}+s)^{\eta}\dz,
\end{align*}
calling in short $\Psi_{2R}(\tau\lambda,Q)$ the set $\{z\in Q:\Psi_{2R}(z)+s>\tau\lambda\}$. Hence, using again H\"older's inequality for Marcinkiewicz spaces, Lemma \ref{lab.Marci}, we have with  $\Psi_{2R}(\mu,Q):=\{z\in Q:\Psi_{2R}(z)+s>\mu\}$
\begin{align*}
\Bigl(\frac{\lambda}{2}\Bigr)^{\eta}-M^{\eta}(\tau\lambda)^{\eta}&\leq\frac{M^{\eta}}{|Q|}\int_{\Psi_{2R}(\tau\lambda,Q)}(\Psi_{2R}+s)^{\eta}\dz\\
&\leq \frac{\gamma M^{\eta}}{\gamma-{\eta}}\frac{|\Psi_{2R}(\tau\lambda,Q)|^{1-\frac {\eta}\gamma}}{|Q|}\times\\
&\quad\times\sup_{\mu>0}\mu^{\eta}\big|\{z\in \Psi_{2R}(\tau\lambda,Q):|\Psi_{2R}(z)|+s>\mu\}\big|^{\frac {\eta}\gamma}\\
&\leq \frac{\gamma M^{\eta}}{\gamma-{\eta}}\biggl[\frac{|\Psi_{2R}(\tau\lambda,Q)|}{|Q|}(\tau\lambda)^{\eta}\\
&\qquad+ \frac{|\Psi_{2R}(\tau\lambda,Q)|^{1-\frac {\eta}\gamma}}{|Q|}\sup_{\mu>\tau\lambda}\mu^{\eta}\big|\Psi_{2R}(\mu,Q)\big|^{\frac {\eta}\gamma}\biggr].
\end{align*}
Again we have been quite sloppy: we refer to \cite[Paragraph 5.4]{BL}. Choosing $\tau$ appropriate:
\[
\frac{1}{2^{\eta}}-M^{\eta}\tau^{\eta}\frac{2\gamma-{\eta}}{\gamma-{\eta}}\geq\frac{1}{4^{\eta}},\qquad\text{i.e.}\qquad \tau=\frac{c(n,p,\nu,L,\gamma)}{M},
\]
we have
\begin{align*}
|Q|&\leq   c\,\frac{|\Psi_{2R}(\tau\lambda,Q)|^{1-\frac {\eta}\gamma}}{(\tau\lambda)^{\eta}}\Big[\sup_{\mu>\tau\lambda}\mu^\gamma\big|\Psi_{2R}(\mu,Q)\big|\Big]^{\frac {\eta}\gamma}\\&\leq c\,(\tau\lambda)^{-\gamma}\sup_{\mu\geq\tau\lambda}\mu^\gamma\big|\Psi_{2R}(\mu,Q)\big|.
\end{align*}
Now we match the previous estimate, which follows if we suppose \eqref{split}$_2$, together with \eqref{meas.1}, which follows from \eqref{split}$_1$ without changes with respect to the case $q<\infty$, we estimate as in \eqref{int.estimate} and then we sum as in Paragraph \ref{level.set}; we get hence
\begin{multline*}
\big|E(A\lambda,Q_{r_1})\big|\leq c\, G(2R,M)\Bigl[\big|E(\lambda/4,Q_{r_2})\big|\\+(\tau\lambda)^{-\gamma}\sup_{\mu\geq\tau\lambda}\mu^\gamma\big|\Psi_{2R}(\mu,Q_{r_2})\big|\Big]. 
\end{multline*}
and also, with the notation introduced after \eqref{notation.k},
\begin{multline*}
 \big|E_k(A\lambda,Q_{r_1})\big|\\\leq c\, G(2R,M)\Bigl[\big|E_k(\lambda/4,Q_{r_2})\big|+(\tau\lambda)^{-\gamma}\sup_{\mu\geq\tau\lambda}\mu^\gamma\big|\Psi_{2R}(\mu,Q_{r_2})\big|\Big].
\end{multline*}
We now multiply inequality \eqref{est.first.par} by $(A\lambda)^\gamma$ and  we take the supremum with respect to $\lambda$ over $(B\lambda_0,\infty)$; this gives, after changing variable
\begin{align*}
\sup_{\lambda>B\lambda_0}(A\lambda)^\gamma&\big|\{z\in Q_{r_1}:{|Du(z)|}_k+s>A\lambda\}\big|\\
&\leq\tilde c\,G(2R,M)\Bigl[\sup_{\lambda>B\lambda_0}\lambda^\gamma\big|\{z\in Q_{r_2}:{|Du(z)|}_k+s>\lambda/4\}\big|\\
&\hspace{3cm}+c(M)\sup_{\lambda>B\tau\lambda_0}\sup_{\mu\geq\lambda}\mu^\gamma\big|\Psi_{2R}(\mu,Q_{r_2})\big|\Bigr].
\end{align*}
Now some easy algebraic manipulations, see again \cite[Paragraph 5.4]{BL}, yield 
\begin{multline*}
\big\|{|Du|}_k+s\big\|_{\mathcal M^{\gamma}(Q_{r_1})}\leq \frac12 \big\|{|Du|}_k+s\big\|_{{\mathcal M}^{\gamma}(Q_{r_2})}+c\, \|\Psi_{2R}+s\|_{\mathcal M^{\gamma}(Q_{r_2})}\\
+c\,R^{\frac N\gamma}\Big(\frac{R}{r_2-r_1}\Big)^{\frac Npd}\bigg[\Bigl(\mean{Q_{r_2}} \big(|Du|+s\big)^p\dz\Bigr)^{\frac dp}+\Bigl(\mean{Q_{r_2}}(\Psi_{2R}+1)^{\eta} \dz \Bigr)^{\frac d{\eta}}\biggr]
\end{multline*}
after choosing $M$ large and $R_0$ small enough to ensure that $G(2R,M)\leq 1/(2^\gamma\tilde c)$ for all $R\leq R_0$. Using one more time Lemma \ref{lab.Marci} we get
\[ 
\biggl(\mean{Q_{2R}}\big(\Psi_{2R}+1\big)^{\eta}\dz\biggr)^{\frac 1{\eta}}\leq \frac{c(p,\gamma)}{(\gamma-p)^{1/{\eta}}}|Q_{2R}|^{-\frac 1\gamma}\|\Psi_{2R}+1\|_{\mathcal M^\gamma(Q_{2R})}
\]
and this finally leads to \eqref{main.est} in the case $q=\infty$.

\subsection*{Acknowledgements}
This research has been supported by the ERC grant 207573 ``Vectorial Problems". The paper was completed while the author was attending the program ``Evolutionary problems'' at the Institut Mittag-Leffler (Djursholm, Sweden) in the Fall 2013; the hospitality of the Institut is gratefully acknowledged.

\end{document}